\documentclass[final]{siamltex}

\usepackage{epsf,latexsym,amsfonts,amssymb,subfigure,mathrsfs,graphicx,stmaryrd,amsmath}
\usepackage[notref,notcite]{showkeys}
\usepackage{color}
\usepackage[notref,notcite]{showkeys}

\newtheorem{remark}[theorem]{Remark}
\numberwithin{equation}{section}
\newcommand{\abs}[1]{\lvert#1\rvert}
\newcommand{\labs}[1]{\left\lvert\,#1\,\right\rvert}

\newcommand{\Lr}[1]{\left(#1\right)}

\newcommand{\integer}[1]{\lfloor\,#1\,\rfloor}
\newcommand{\fraction}[1]{\{\,#1\,\}}
\newcommand{\mc}[1]{\mathcal #1}
\newcommand{\mb}[1]{\mathbb #1}
\def\negint{{\int\negthickspace\negthickspace\negthickspace
\negthinspace -}}
\newcommand{\wt}[1]{\widetilde{#1}}

\def\Lqc{\mc{L}_{\text{qc}}}
\def\Lat{\mc{L}_{\text{at}}}
\def\Lcb{\mc{L}_{\text{CB}}}

\def\al{\alpha}
\def\del{\delta}
\def\eps{\varepsilon}
\def\ga{\gamma}
\def\ka{\kappa}
\def\lam{\lambda}
\def\om{\omega}

\def\md{\mathrm d}
\def\dx{\,\md x}
\def\dy{\,\md y}
\def\I{\imath}
\def\nn{\nonumber}
\def\Del{\Delta}
\def\vphi{\varphi}

\title{On the effect of ghost force in the quasicontinuum method: Dynamic problems in one dimension}
\author{Xiantao Li
\thanks{Department of Mathematics, the Pennsylvania State University, University Park,
Pennsylvania, 16802, U.S.A. ({\tt xli@math.psu.edu})
{The work of Li was supported by National Natural Science
              Foundation grant DMS1016582.}}
\and{Pingbing Ming}
\thanks{LSEC, Institute of Computational
  Mathematics and Scientific/Engineering Computing,
  AMSS, Chinese Academy of Sciences,
  No. 55, Zhong-Guan-Cun East Road,
  Beijing 100190, China. ({\tt mpb@lsec.cc.ac.cn}) {The work of Ming was supported by National Natural Science Foundation of China grant 10932011, and by the funds from Creative Research Groups of China through grant 11021101, and by the support of CAS National Center for Mathematics and Interdisciplinary Sciences.}}}
\begin{document}

\maketitle

\begin{abstract}
Numerical error induced by the ``ghost forces" in the quasicontinuum method is studied in the context of dynamic problems. The error in the \({ W}^{1,\infty}\) norm is analyzed for the time scale \(\mc{O}(\eps)\) and the time scale \(\mc{O}(1)\) with $\eps$ being the lattice spacing.
\end{abstract}


\section{Introduction}

The present paper is mainly concerned with the error produced by the ghost forces in
quasicontinuum (QC) type of multiscale
coupling methods for crystalline solids. In these methods, one reduces the degrees of freedom of an atomic level description by replacing part of the system with
continuum mechanics models~\cite{TadmorOrtizPhillips:1996, BeXi03, KnapOrtiz:2001, MillerTadmor:2009, ShenoyMillerTadmorPhillipsOrtiz:1999}.
Such integrated methods have been very useful in studying mechanical
properties of lattice defects. It allows one to simulate a relatively large system while still able to keep
the atomistic description around critical areas, such as crack tips and dislocations cores.
These methods have also drawn a great deal of attention from numerical analysts. We refer to~\cite{Lin03, EMing:2005, ELuYang:2006, MingYang:2009, DobsonLuskin:2009a,
DobsonLuskinOrtner:2010a, Shapeev:2011, Ortner} and references therein for
a partial list of the representative work.
Nevertheless, many challenges in the analysis
of these methods still remain. Examples include high-dimensional
problems, systems with line or wall defects, and solutions near bifurcation points.
We refer to~\cite {LuMing:2011, OrtnerShapeev}
for a review of the state-of-art of this field. A critical issue that arises in the numerical
analysis is the {\em ghost force},  which is the non-zero forces on the atoms near the interface at the equilibrium state~\cite{ShenoyMillerTadmorPhillipsOrtiz:1999}.
For statics problems, the elimination of ghost forces has been a necessary ingredient to achieve
uniform accuracy~\cite{ELuYang:2006, Ortner}.

In the static case, the QC method couples a molecular statics model
to the Cauchy-Born elasticity model.
For one-dimensional models, the influence of ghost forces has been explicitly
characterized in~\cite{Ming:2008,MingYang:2009, DobsonLuskin:2009a}. They
found that the ghost force may lead to an $\mc{O}(1)$ error for the gradient of
the solution, and  the width of the resulting interfacial layer is of
the size $\mc{O}(\eps\abs{\ln\eps})$,
where $\eps$ is the equilibrium bond length.
The influence of the ghost force for a two-dimensional model with planar interface has recently been studied in~\cite {ChenMing:2012}.
It was found that the ghost forces still lead to an $\mc{O}(1)$ error for the gradient of
the solution, while the interfacial layer caused by the ghost force is of the size $\mc{O}(\sqrt\eps)$,
which is much wider than that of the one-dimensional problems.

The QC method can be extended to dynamic problems using the
coarse-grained energy and the Hamilton's principle
\cite{Shenoy1999,Rodney03}. { The dynamic QC method}
couples an elastodynamics
model with a molecular dynamics model. Many dynamic coupling methods with similar goals have been developed~\cite{AbBrBe98, BeXi03, EHu02, EHu01,ELi05, LiYaE09, TaHoLi06, TaHoLi06b, WaLi03, XiBe04} ever since.
However, very little has been done to address the stability and accuracy
of these methods.
 Most numerical studies have been focused on the
artificial reflections at the interface. The reflection is caused by the drastic
change in the
dispersion relation across the interface, which is often due to the
difference between the mesh size in the continuum region and the lattice
spacing in the molecular dynamics model. The reflection can be studied by considering an incident wave
packet traveling toward the interface and examine the amplitude of the reflected waves. The issue of ghost forces, however, has not yet been addressed.

The purpose of this paper is to study the effect of ghost forces in the
context of {\it dynamic problems}. Motivated by the results in the static case, we
expect that ghost forces will continue to play an important role in
dynamic coupling models. To focus primarily on the issue of ghost forces, we
consider the dynamic model~\cite{Shenoy1999,Rodney03} derived from the original QC method when the mesh size coincides with the lattice spacing.
In addition, the initial displacement is given by a uniform deformation. This
allows us to compute the error caused only by the ghost forces.
The error will be studied in the \(W^{1,\infty}\)-norm as in the static problem~\cite{Ming:2008,DobsonLuskin:2009a, MingYang:2009, ChenMing:2012}.
Our study shows that the error, which is
initially zero, grows very quickly, and already becomes $\mc{O}(1)$ at the
time scale \(\mc{O}(\eps)\) . The error exhibits fast oscillations,
with amplitude on the order of
\(\eps\). On the time scale \(\mc{O}(1)\), which is typically the time
scale of interest, the amplitude of the oscillations grows, and it is
bounded by an \(\mc{O}(\sqrt\eps)\) quantity. The average of the
oscillations has a peak at the interface. In contrast to the static
case, where the error is mainly concentrated at the interface, the error
in the dynamic case is observed in the entire domain on the time scale 
\(\mc{O}(1)\). These observations are quite different from those of wave 
reflections, and it indicates that the effect of ghost forces is a separate numerical issue.

The rest of the paper is organized as follows. In Section 2 we
describe the one-dimensional atomistic model
and the derivation of the QC model, and
briefly demonstrate the appearance of the ghost forces. In Section 3 we
show results from several numerical tests. They provide some insight into
the evolution of the error.
The next three sections are devoted to the analysis of the error for
short and long time scales. We draw some conclusions in the last section.
\section{Motivation and the Formulation of the Problem}\label{sec2}
As in~\cite {EMing:2007b}, we consider the dynamic problem of a
one-dimensional chain of atoms. The interatomic interaction is assumed to be among
the nearest and the next nearest neighbors.
Let \(x\) be the reference position of an atom, and \(\wt{y}(x,t)\) be the current position at time \(t\). The equations of motion for the atoms in the chain read,
\begin{equation}\label{eq: md}
\left\{
\begin{aligned}
\ddot{\wt{y}}(x,t)-\Lat[\wt{y}](x,t)&=0,\qquad x\in\mb{L},\\
\wt{y}(x,0)&=x,\;\dot{\wt{y}}(x,0)=x,\\
\wt{y}(x,t)-x&\; \text{is periodic with period}\;1.
\end{aligned}\right.
\end{equation}
Here, we have set the mass to unity, and \(\mb{L}\equiv\{ j\varepsilon, j \in \mathbb{N}\}\cap (-1/2, 1/2)\) with \(\eps\) being the lattice parameter.
The operator \(\Lat\) is defined as
\begin{align}\label{eq: Lat}
\Lat[z](x,t)&\equiv
\eps^{-2}\bigl [\ka_2 z(x-2\eps,t)+\ka_1
z(x-\eps,t)+\ka_1 z(x+\eps,t)+\ka_2 z(x+2\eps,t)\nn\\
&\phantom{\equiv
\eps^{-2}\bigl [}\quad-2(\ka_1+\ka_2)z(x,t)\bigr].
\end{align}

Since the issue of ghost force arises even for harmonic interaction, we consider here a linear model, which
can be considered as a harmonic approximation of a fully nonlinear model.
In~\eqref{eq: Lat}, $\ka_1$ and $\ka_2$ are the force constants computed from an interatomic potential. For example, for a pair potential, the energy is given by
\[
  E= \sum_{x}\Biggl[\varphi\Lr{\frac{y(x+\eps)-y(x)}{\eps}}+ \varphi\Lr{\frac{y(x+2\eps)-y(x)}{\eps}}\Biggr].
\]
Direct calculation yields
\[
  \ka_1= \varphi''(1), \;\ka_2=\varphi''(2).
\]

One commonly used model is the Lennard-Jones potential~\cite {LennardJones:1924}:
\[
\varphi(r)=\Lr{\sigma/r}^{12}-\Lr{\sigma/r}^6.
\]
If only the nearest and the next nearest neighborhood interactions are considered, the lattice parameter is given by
\[
\eps=2^{1/6}\Lr{\dfrac{1+2^{-12}}{1+2^{-6}}}^{1/6}\sigma.
\]
In this case, the force constants are
\[
\ka_1=156C^2-42C\approx 18.886 \qquad
\text{and}\qquad\ka_2=2^{-6}(156C^22^{-6}-42C)\approx -0.323,
\]
where \(C=(1+2^{-6})/[2(1+2^{-12})]\).
The above formula has also appeared in~\cite{HouchmandzadehtLajzerowicztSalje:1992}. Notice that the second force constant is negative, but it is much smaller than the first force constant in magnitude.
\medskip

With the harmonic approximation, the potential energy takes the form of
\[
  E = \sum_{x}\Biggl[\frac{\ka_1}{2}\Lr{\frac{y(x+\eps)-y(x)}{\eps}}^2 + \frac{\ka_2}{2}\Lr{\frac{y(x+2\eps)-y(x)}{\eps}}^2\Biggr].
  \]
The dynamic model~\eqref{eq: md} can be derived from this energy using Hamilton's principle. Notice that
the energy can be divided into the energy at each atom site: i.e. \(E= \sum_x E(x)\), in which
\[
 \begin{aligned}
  E(x)&=\frac{\ka_2}{4} \Lr{\frac{y(x+2\eps)-y(x)}{\eps}}^2
  +\frac{\ka_1}{4} \Lr{\frac{y(x+\eps)-y(x)}{\eps}}^2 \\
  &\quad+\frac{\ka_1}{4} \Lr{\frac{y(x)-y(x-\eps)}{\eps}}^2
  +\frac{\ka_2}{4} \Lr{\frac{y(x)-y(x-2\eps)}{\eps}}^2.
  \end{aligned}
  \]

In the QC method, one defines a local region where the atomistic model is approximated
by the Cauchy-Born elasticity model~\cite {BornHuang:1954}.
One also defines a nonlocal region where the atomistic description is kept.
Without loss of generality, we assume that the interface is located at \(x=0\) and the
nonlocal region is in the domain \(x<0\). We further assume that the mesh size is equal to the lattice parameter to primarily focus on the effect of ghost forces.  The Cauchy-Born approximation of the energy in the local region is given by
\[
E_{{\rm CB}}(x)= \frac{\ka_1+4\ka_2}{4}\Lr{\frac{y(x+\eps)-y(x)}{\eps}}^2 +
   \frac{\ka_1+4\ka_2}{4}\Lr{\frac{y(x)-y(x-\eps)}{\eps}}^2.
\]
At the interface \(x=0\), the energy takes a mixed form:
\[
\begin{aligned}
E(0) &= \dfrac{\ka_2}{4} \Lr{\dfrac{y(x)-y(x-2\eps)}{\eps}}^2
+\dfrac{\ka_1}{4} \Lr{\frac{y(x)-y(x-\eps)}{\eps}}^2  \\
&\quad+\dfrac{\ka_1+4\ka_2}{4} \Lr{\dfrac{y(x+\eps)-y(x)}{\eps}}^2.
\end{aligned}
\]

With such energy summation rule, we may write the QC approximation of \(\Lat\) as \(\Lqc\), with \(\Lqc\) given below.
For $x\le -2\eps$,
\begin{align*}
\Lqc[z](x,t)&\equiv
\eps^{-2}\bigl [\ka_2 z(x-2\eps,t)+\ka_1
z(x-\eps,t)+\ka_1 z(x+\eps,t)+\ka_2 z(x+2\eps,t)\\
&\phantom{\equiv
\eps^{-2}\bigl [}\quad-2(\ka_1+\ka_2)z(x,t)\bigr],
\end{align*}
and for $x\ge 2\eps$,
\[
\Lqc[z](x,t)= \Lcb, \; \Lcb\equiv
\eps^{-2}(\ka_1+4\ka_2)[z(x-\eps,t)-2z(x,t)+z(x+\eps,t)].
\]
This is exactly the operator corresponding to the Cauchy-Born approximation.

At the interface, we have for $x=-\eps$,
\begin{align*}
\Lqc[z](x,t)&\equiv
\eps^{-2}\bigl[\ka_2 z(x-2\eps,t)+\ka_1 z(x-\eps,t)+\ka_1 z(x+\eps,t)+\dfrac{\ka_2}2
z(x+2\eps,t)\\
&\phantom{\equiv
\eps^{-2}\bigl[}\quad-(2\ka_1+3\ka_2/2)z(x,t)\bigr],
\end{align*}
for $x=0$,
\begin{align*}
\Lqc[z](x,t)&\equiv
\eps^{-2}\Bigl[\ka_2 z(x-2\eps,t)+\ka_1 z(x-\eps,t)+(\ka_1+4\ka_2)z(x+\eps,t)\\
&\phantom{\equiv\eps^{-2}\Bigl[}\quad-(2\ka_1+5\ka_2)z(x,t)\Bigr],
\end{align*}
and for $x=\eps$,
\begin{align*}
\Lqc[z](x,t)&\equiv
\eps^{-2}\Bigl[\dfrac{\ka_2}2z(x-2\eps,t)+(\ka_1+4\ka_2)z(x-\eps,t)
+(\ka_1+4\ka_2)z(x+\eps,t)\\
&\phantom{\eps^{-2}\Bigl[}\qquad-(2\ka_1+17\ka_2/2)z(x,t)\Bigr].
\end{align*}

Using the Hamilton's principle, we write the QC model as
\begin{equation}\label{eq: qc}
\left\{
\begin{aligned}
\ddot{\wt{y}}(x,t)-\Lqc[\wt{y}](x,t)&=0,\qquad x\in\mb{L},\\
\wt{y}(x,0)=x,\;\dot{\wt{y}}(x,0)&=x,\\
\wt{y}(x,t)-x\; &\text{is periodic with period}\; 1.
\end{aligned}\right.
\end{equation}

The initial and boundary conditions have been chosen as a uniform deformation
 in order to identify the effect of the ghost force.  We will compute the deviation of the solution away from the equilibrium. For this purpose, we
define the error $y(x,t)=\wt{y}(x,t)-x$, and we have,
\begin{equation}\label{eq: qcd}
\begin{aligned}
\ddot{y}(x,t)-\Lqc[y](x,t)&=\ddot{\wt{y}}(x,t)-\Lqc[\wt{y}-x](x,t)\\
&=\Lqc[x](x,t)\equiv f(x,t),
\end{aligned}
\end{equation}
with \(f\) given explicitly by
\begin{equation}\label{eq: gf}
f(x,t)=\left\{
\begin{aligned}
0\quad &\text{if\quad}\abs{x}\ge 2\eps,\\
-\dfrac{\ka_2}{\eps}\quad&\text{if\quad}x=-\eps,\\
\dfrac{2\ka_2}{\eps}\quad&\text{if\quad}x=0,\\
-\dfrac{\ka_2}{\eps}\quad&\text{if\quad}x=\eps.
\end{aligned}\right.
\end{equation}
The function \(f(x,t)\) is precisely the ghost force.
Since it is independent of the temporal variable, we denote it
by $f(x)$ for simplicity. Finally, we supplement the above problem
with the homogeneous initial condition and periodic boundary condition as
\begin{equation}\label{eq:ibc}
\left\{\begin{aligned}
y(x,0)&=0\text{\quad and\quad}\dot{y}(x,0)=0,\qquad x\in\mb{L}.\\
y(x,t)&\; \text{is periodic with period\;} 1.
\end{aligned}\right.
\end{equation}
\section{Observations from Numerical Results}
Since the operator \(\Lqc\) coincides with \(\Lcb\) in the local region, and with \(\Lat\) in the
nonlocal region, it is natural to look at models similar to~\eqref{eq: qcd}, in which
\(\Lqc\) is replaced by either
\(\Lcb\) or \(\Lat\) in the entire domain. Therefore, our numerical experiments are conducted for the following three models:
\begin{itemize}
\item Model I. \(\Lqc\) is approximated by \(\Lcb\):
\(  \ddot{y} - \Lcb[y]= f. \)
\item Model II. \(\Lqc\) is approximated by \(\Lat\):
\(  \ddot{y} - \Lat[y]= f.\)
\item Model III: The quasicontinuum model \eqref{eq: qcd}.
\end{itemize}
In all these models, we impose homogeneous initial condition and periodic boundary condition~\eqref{eq:ibc}.

As an example, the force constants are obtained from the Morse potential~\cite {Morse:1929}. In particular, we choose
\(\ka_1= 4.4753\) and \(\ka_2= 0.4142\).
All the simulations are performed in the domain \(x\in [-1/2, 1/2]\), and the ODEs are integrated using the Verlet's method. Since all three models are Hamiltonian systems, this method is particularly suitable.

We first show the solutions computed from the three models at different time step.  The results are shown in Fig. \ref{fig: dy_x}. For this set of numerical tests, we have chosen \(\eps=1/2000\). We observe that the error first developed at the interface, and then it starts to spread toward the local and nonlocal region for all three models. Another noticeable
feature is that the error exhibits a peak at the interface, and the peak remains for all later time. At \(t=1\), the error is observed in the entire domain.
\begin{figure}[htbp]
\begin{center}
\includegraphics[width=1.5in,height=4in]{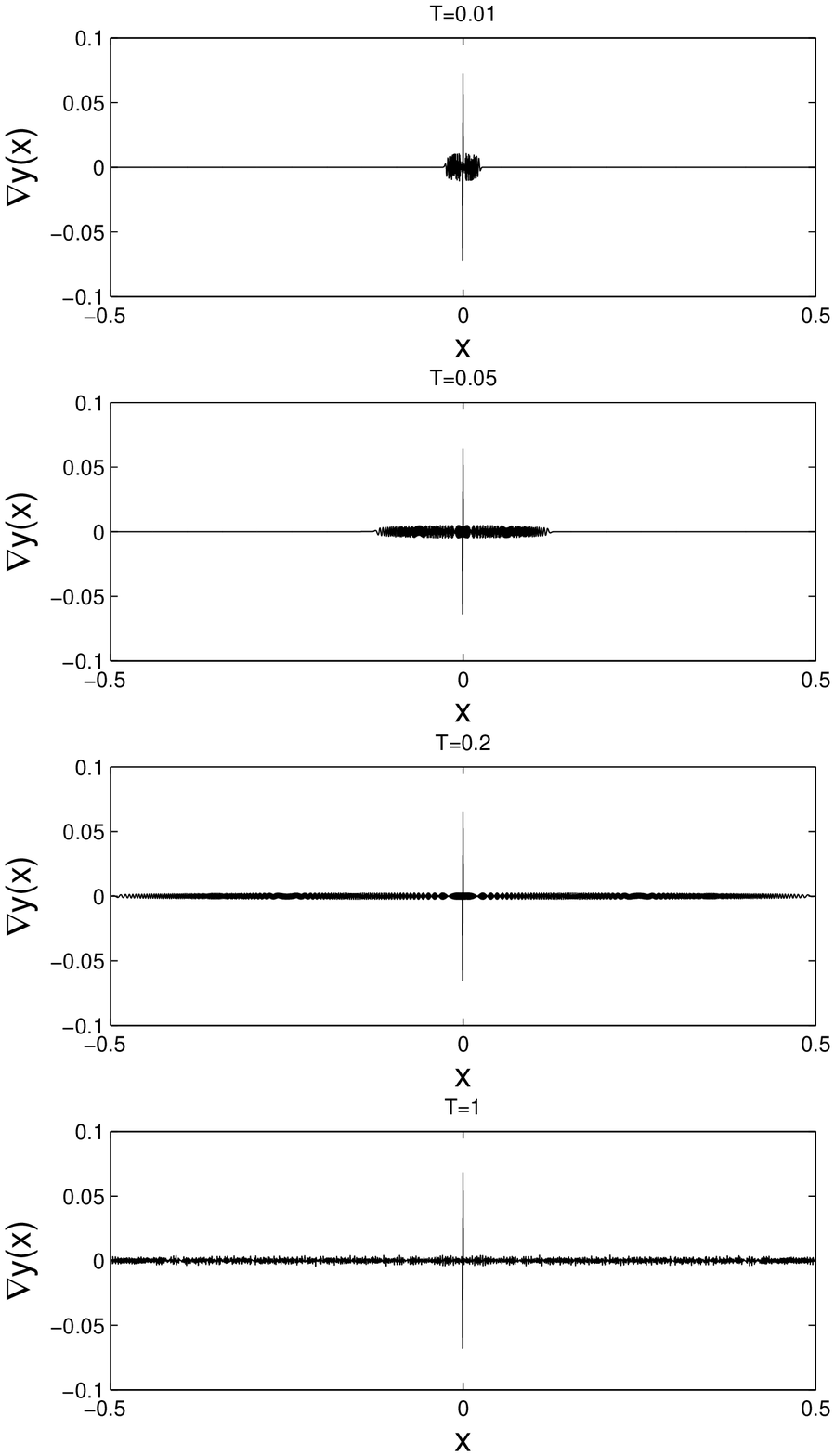}
\includegraphics[width=1.5in,height=4in]{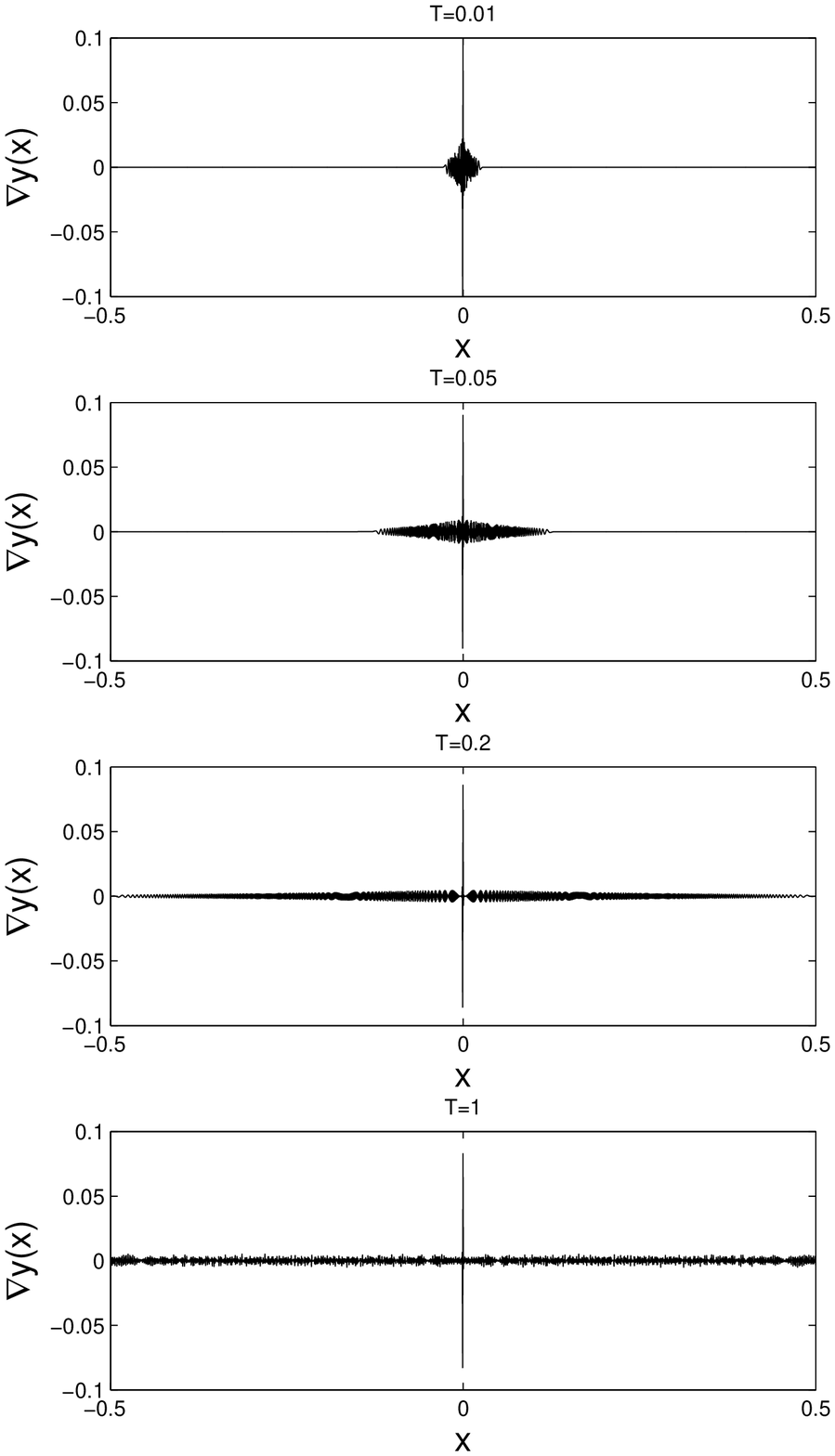}
\includegraphics[width=1.5in,height=4in]{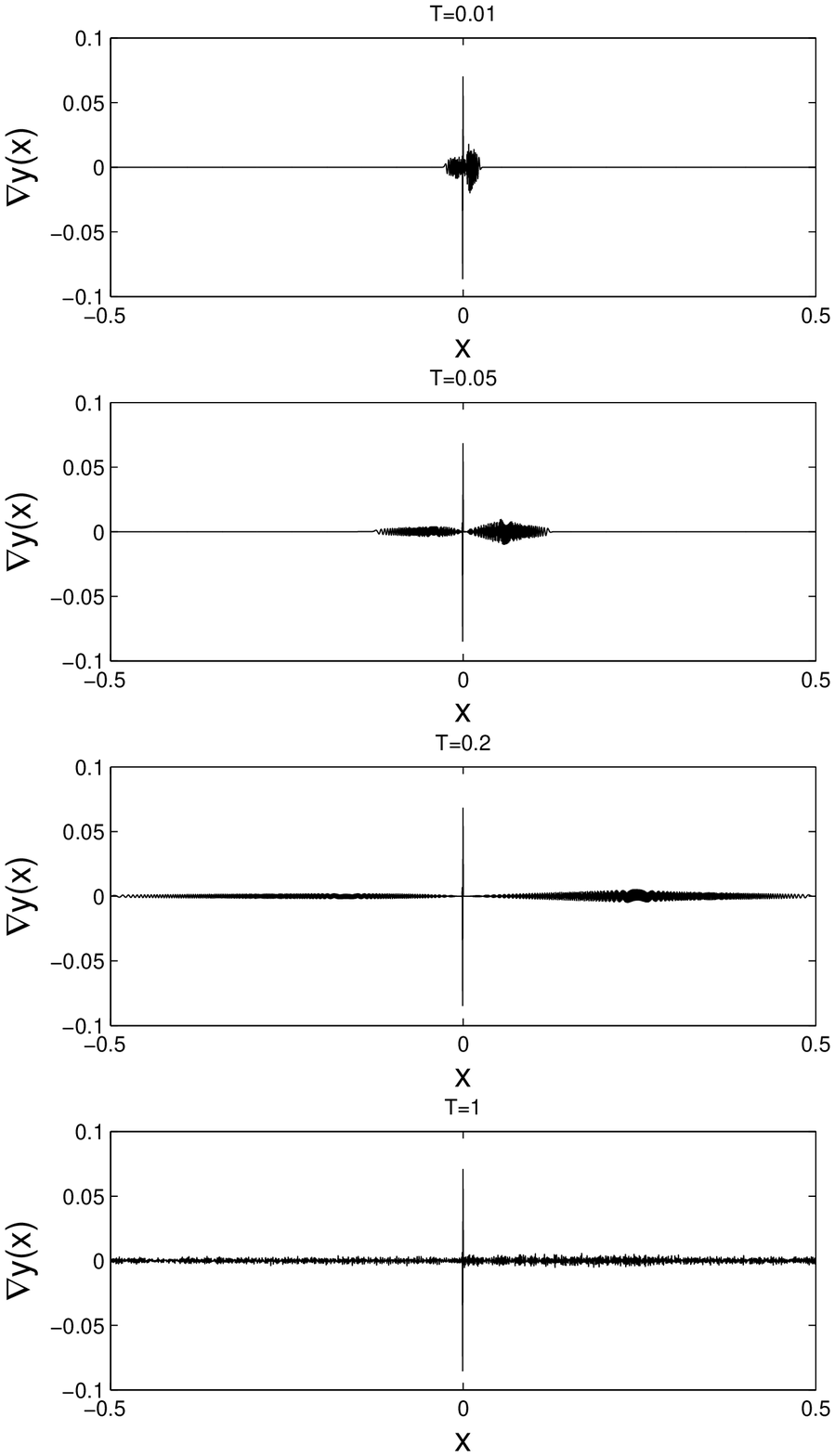}
\caption{The gradient of the error. Left to right: Solution computed from Model I, II
and III. From top to bottom: The solutions at time \(t=0.01, 0.05, 0.2 \;\rm{and}\; 1\). }
\label{fig: dy_x}
\end{center}
\end{figure}
Our main observations here can be summarized as following: (1) In the presence of ghost forces, the error grows very quickly. It reaches  \(\mc{O}(1)\) on the time scale of \(\mc{O}(\eps)\); (2) At the time scales of \(t=\mc{O}(\eps)\) and \(t= \mc{O}(1)\),
the solutions of all three models are qualitatively the same.

Next we monitor the solution for those atoms near the interface. In Fig. \ref{fig: dy_t}, we show the time history for those atoms. We observe that for most of the time, the error oscillates around certain constant values, and the constant values depend on the location of the atom. These constant values show a peak at the interface \(x=0\).
\begin{figure}[htbp]
\begin{center}
\includegraphics[width=1.5in,height=4in]{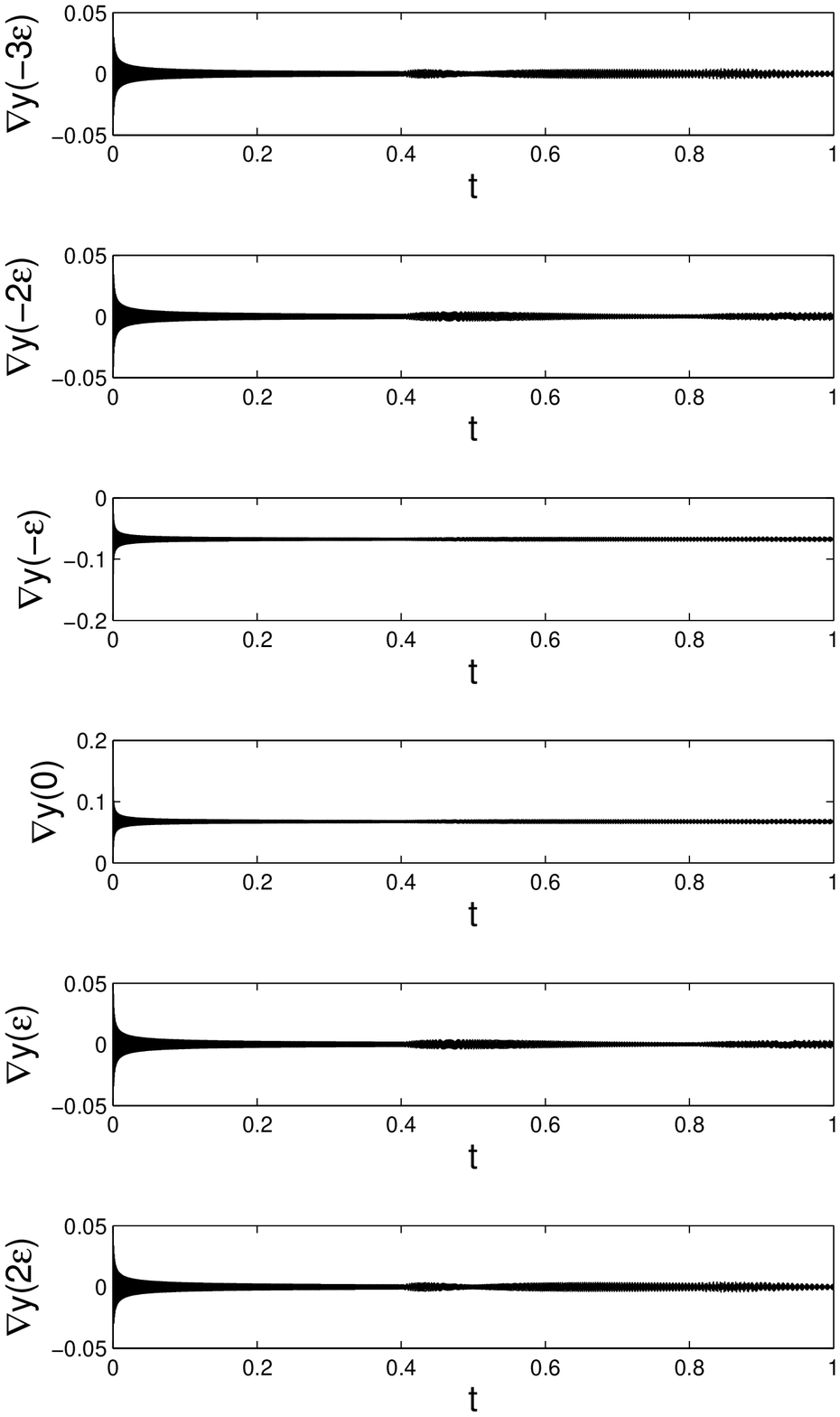}
\includegraphics[width=1.5in,height=4in]{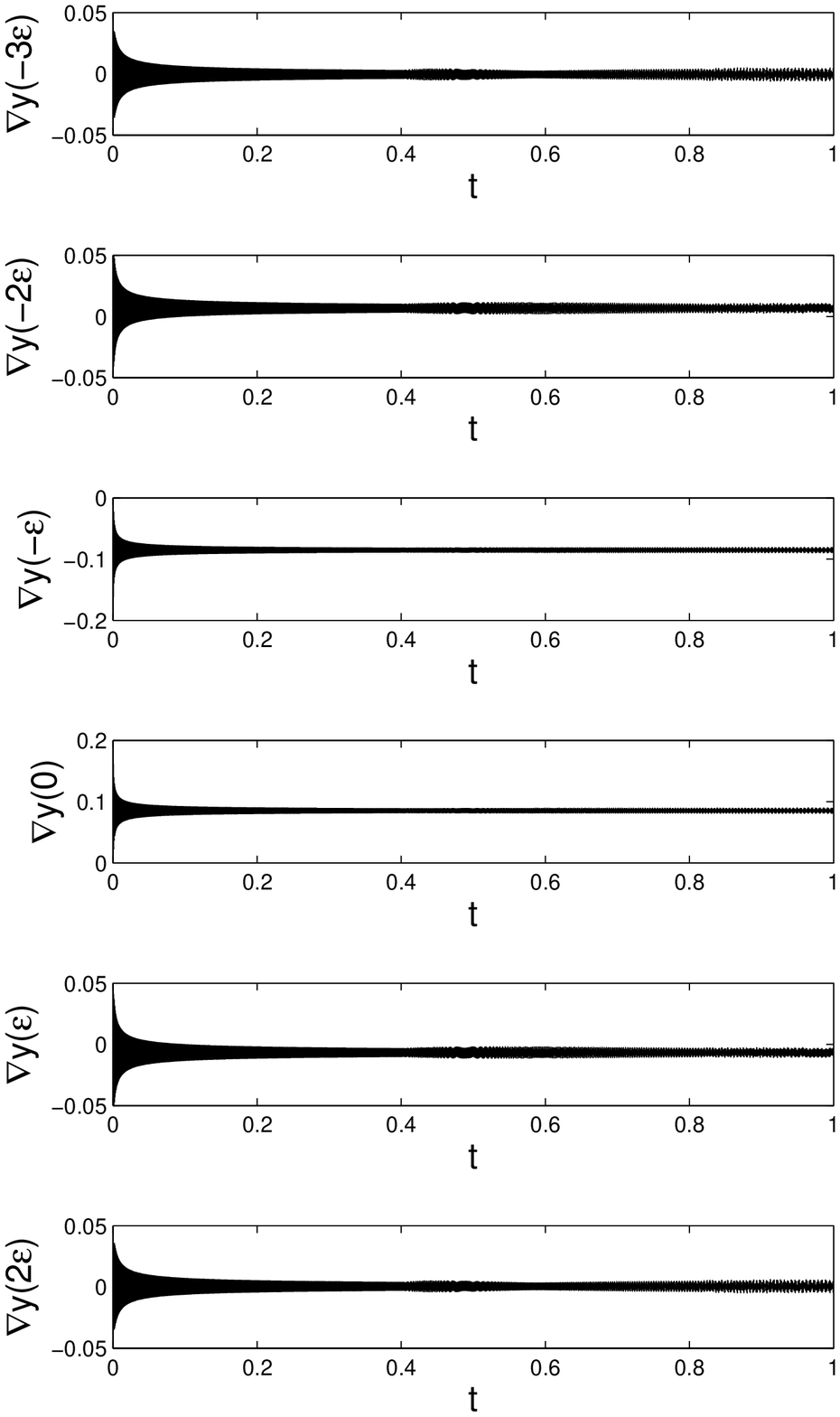}
\includegraphics[width=1.5in,height=4in]{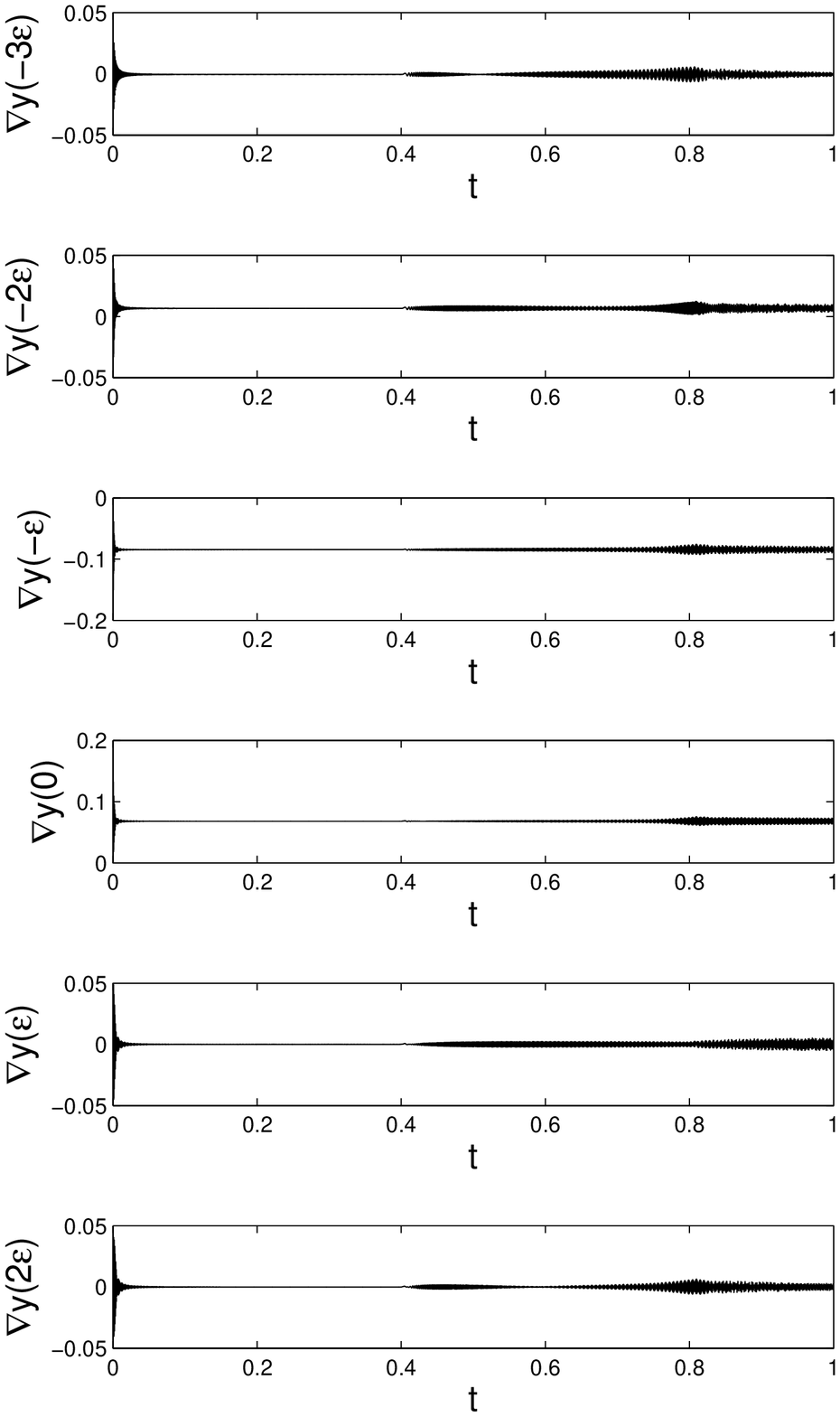}
\caption{The time history of the gradient of the error near the interface. Left to right: Solution computed from Model I, II and III. From top to bottom: The solutions near the interface: \(x=-3\eps\),\(x=-2\eps\),\(x=-\eps\),\(x=0\),\(x=\eps\) and \(x=2\eps\). }
\label{fig: dy_t}
\end{center}
\end{figure}

In the last two figures, Fig.~\ref{fig: dy_t_cb_N} and Fig.~\ref{fig: dy_t_qc_N}, we show the time history of the solution at the interface for various values of \(\eps\). The main observation is that the amplitude of the oscillation decrease as \(\eps\) gets small. However, the constant values around which the solutions oscillate do not change as \(\eps\) varies.

\begin{figure}[htbp]
\begin{center}
\includegraphics[width=1.5in,height=4in]{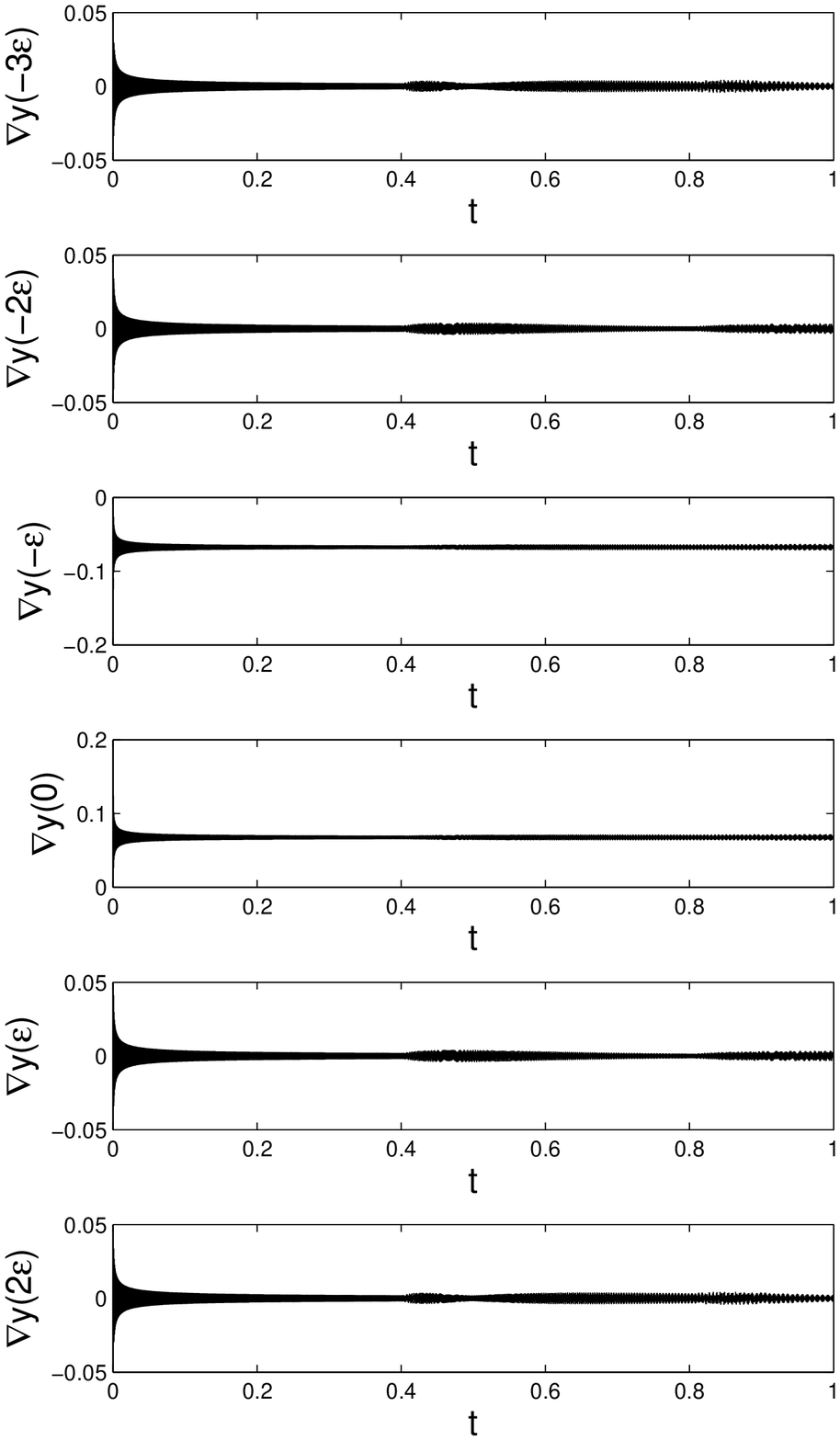}
\includegraphics[width=1.5in,height=4in]{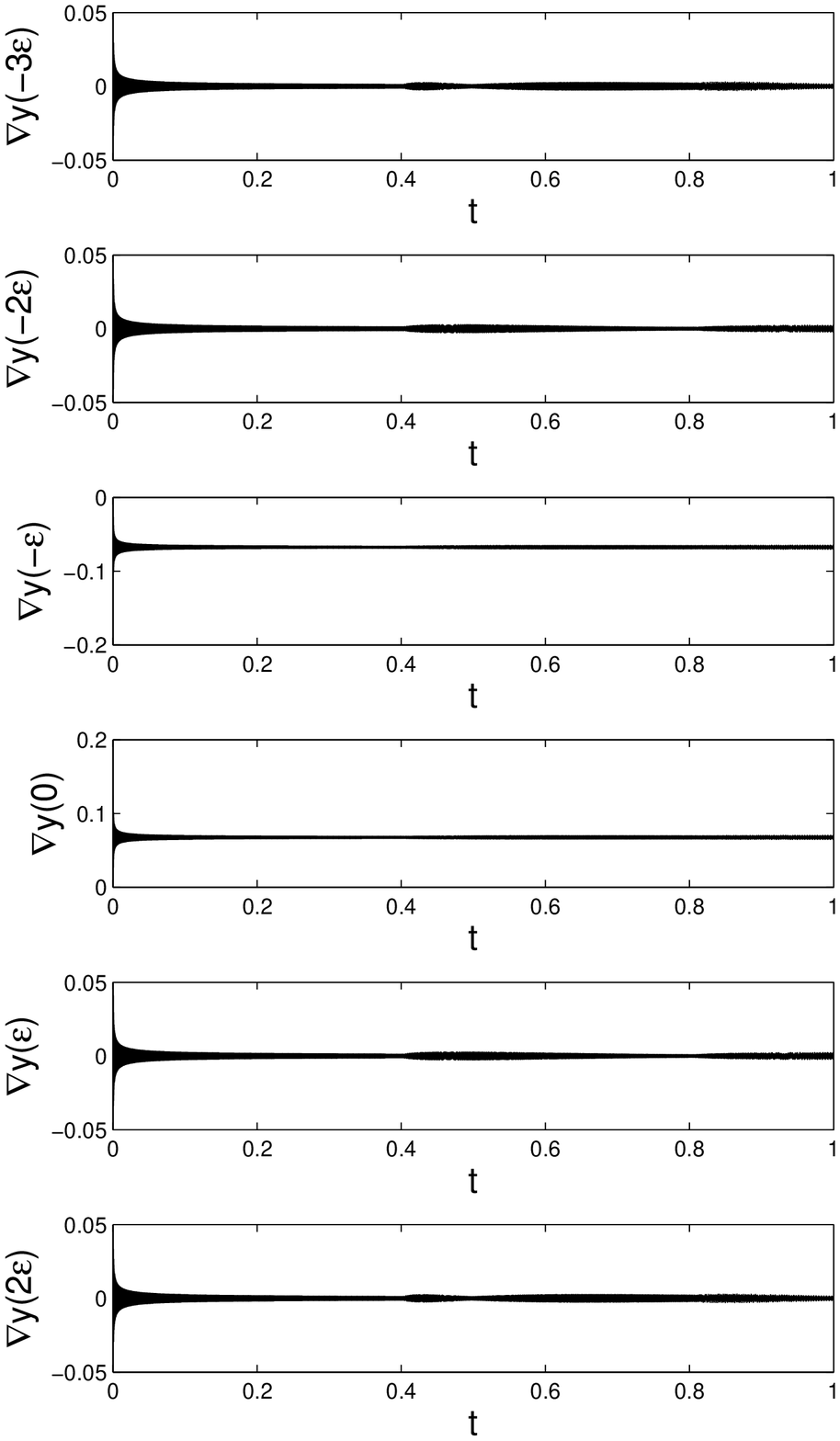}
\includegraphics[width=1.5in,height=4in]{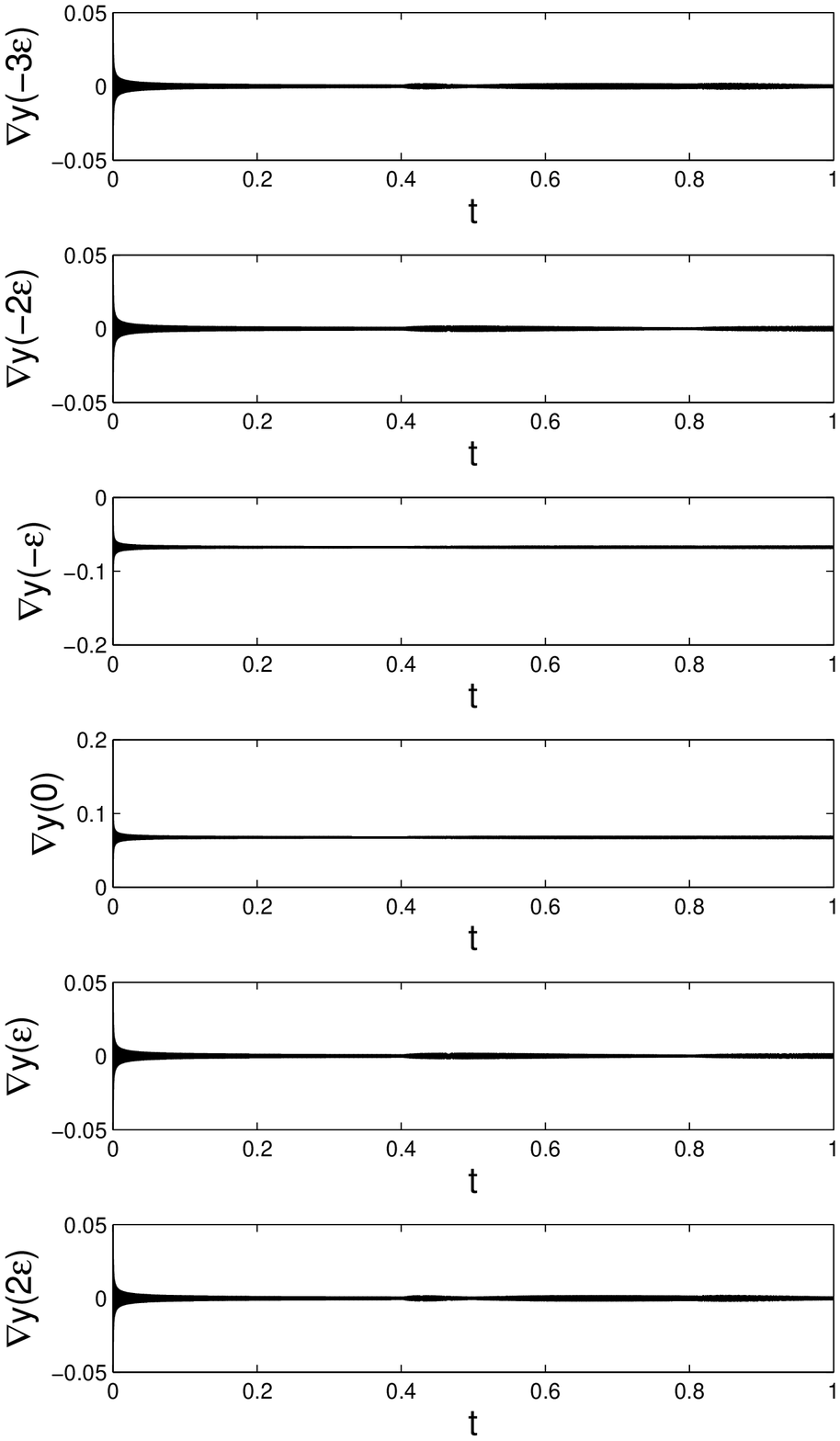}
\caption{The gradient of the error. The solutions are computed from Model I with
different choice of \(\epsilon\). Left to right: Solution computed for \(\eps=1/2000\),
\(\eps=1/4000\) and \(\eps=1/8000\).
From top to bottom: The solutions near the interface: \(x=-3\eps\), \(x=-2\eps\), \(x=-\eps\), \(x=0\), \(x=\eps\) and \(x=2\eps\). }
\label{fig: dy_t_cb_N}
\end{center}
\end{figure}

\begin{figure}[htbp]
\begin{center}
\includegraphics[width=1.5in,height=4in]{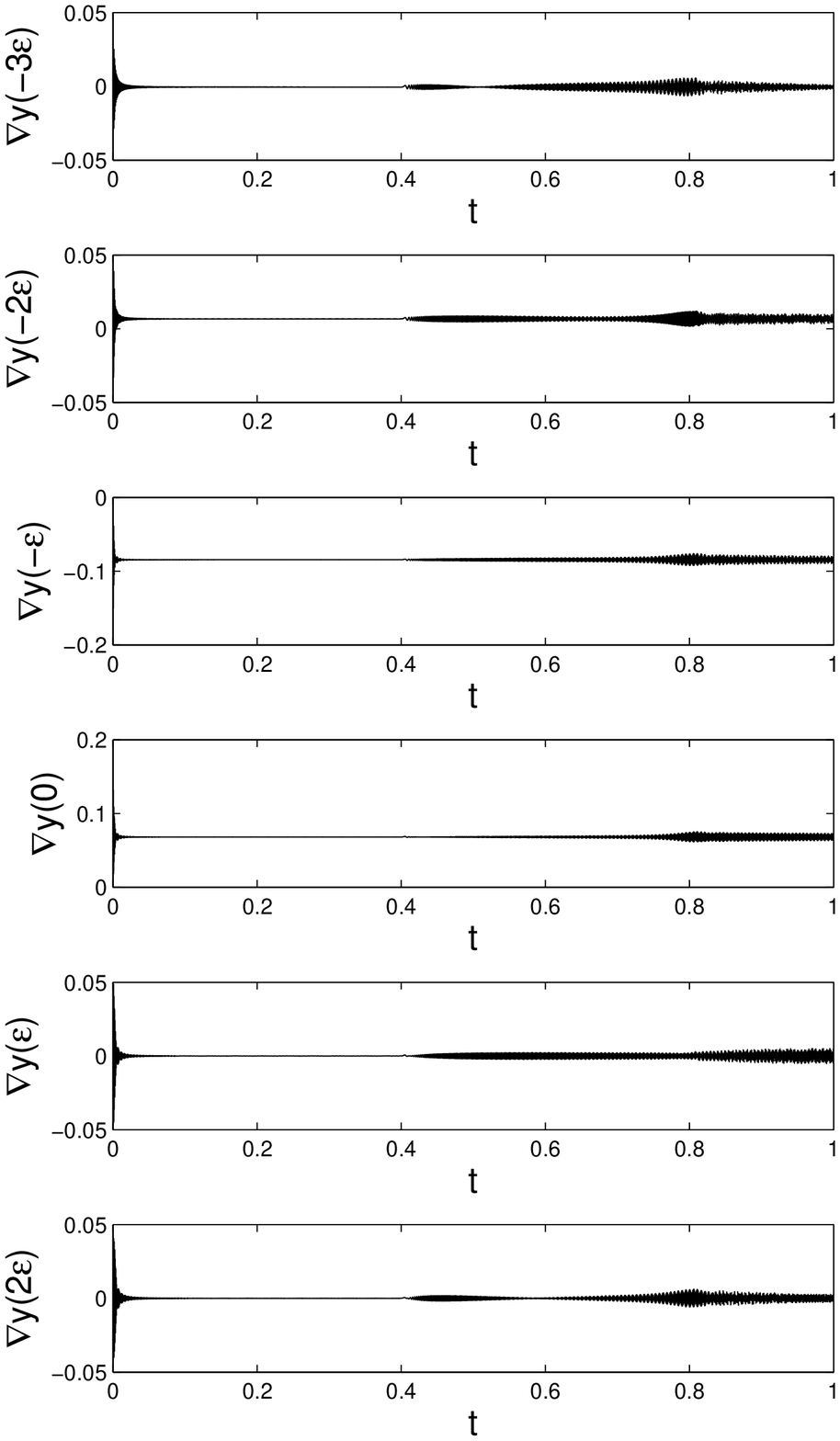}
\includegraphics[width=1.5in,height=4in]{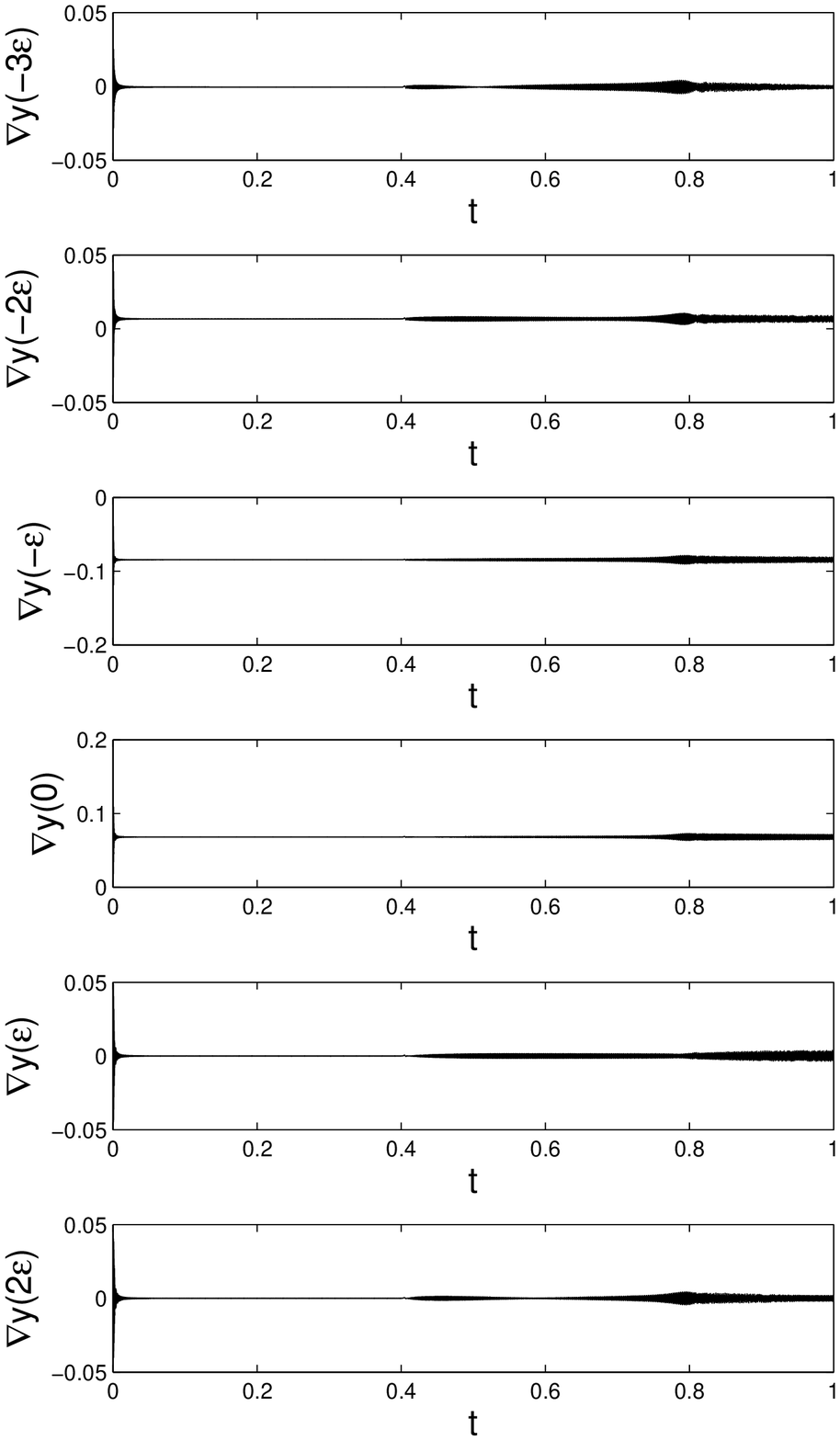}
\includegraphics[width=1.5in,height=4in]{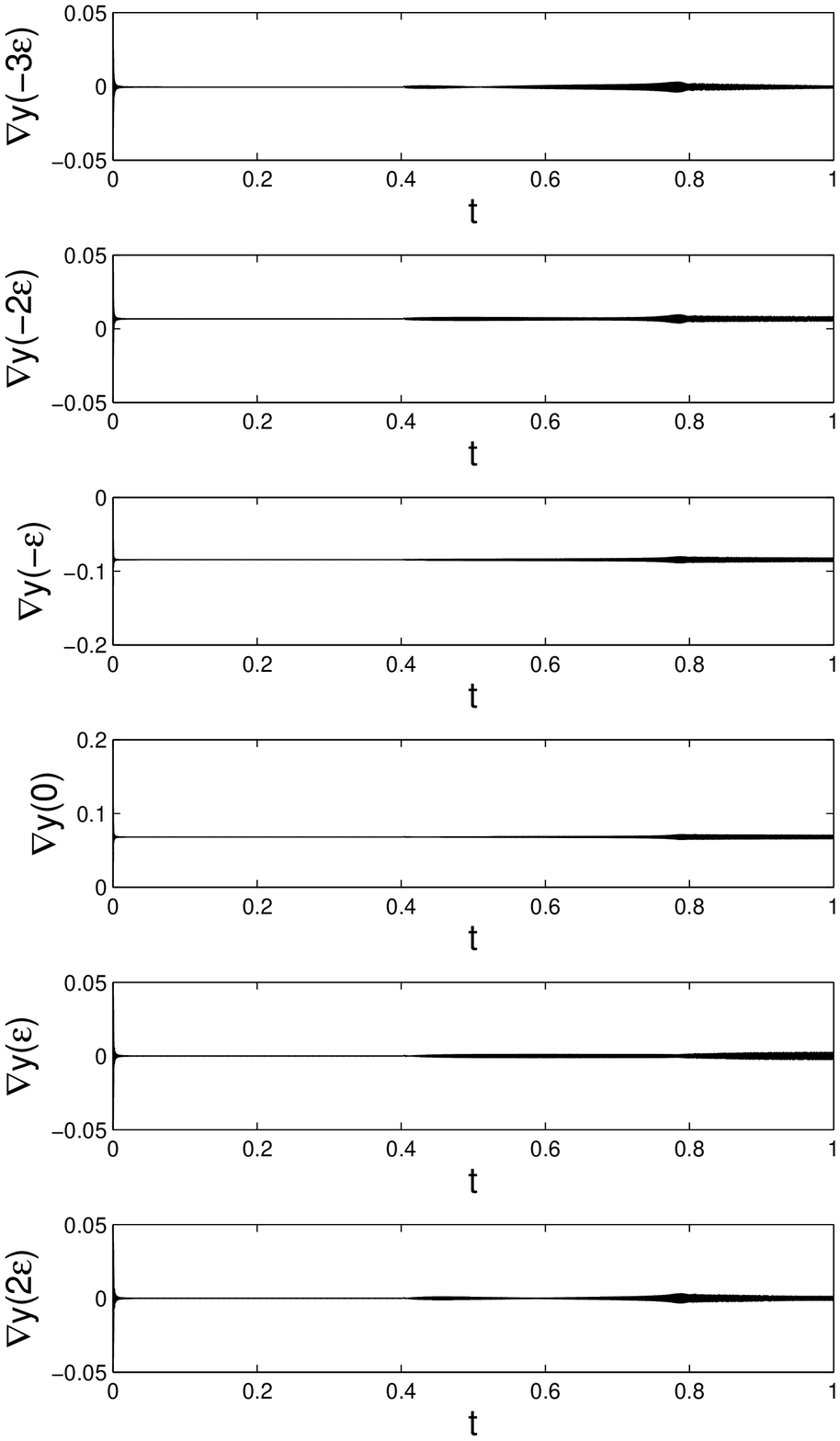}
\caption{The gradient of the error. The solutions are computed from Model III with
different choices of \(\eps\). Left to right: Solution computed for \(\eps=1/2000\),
\(\eps=1/4000\) and \(\eps=1/8000\).
From top to bottom: The solutions near the interface: \(x=-3\eps\), \(x=-2\eps\), \(x=-\eps\), \(x=0\), \(x=\eps\) and \(x=2\eps\). }~\label{fig: dy_t_qc_N}
\end{center}
\end{figure}
\section{Explicit Solutions for the Approximating Model}
In view of the numerical results, it seems that the solution of model I bears similarity to the dynamical behavior of the original problem \eqref{eq: qcd} on the time scale $\mc{O}(\eps)$ and time scale $\mc{O}(1)$.  Therefore we will turn to this model to study the effect of ghost forces. Model I is convenient to analyze, particularly because it admits explicit solutions of a simple form. Recall that in Model I, we solve the following problem,
\begin{equation}\label{eq:reducemodel}
\left\{
\begin{aligned}
\ddot{y}(x,t)-\Lcb[y](x,t)&=f(x),\qquad x\in\mb{L},\\
y(x,0)=0,\dot{y}(x,0)&=0,\\
y(x,t)&\text{\quad is periodic with period}\;1.
\end{aligned}\right.
\end{equation}
Without loss of generality, we let $\mb{L}=(-1/2,1/2]$ with $N$ atoms. We assume
that $N$ is an even integer for technical simplicity. Obviously, $\eps=1/N$. We will switch to the notation that,
\begin{equation}\label{eq:number}
w(n,t)= w(-1/2+n\eps,t),\quad n=1, 2, \cdots, N.
\end{equation}

We now express the solution of~\eqref{eq:reducemodel} in an explicit form. To begin with, we consider the lattice Green's function,  which is
defined as the solution of the following problem:
\begin{equation}\label{eq:latgreen}
\left\{
\begin{aligned}
\ddot{G}(n,t)-\Lcb[G](n,t)&=0,\qquad n=0,\dots,N,\\
G(n,0)=0,\dot{G}(n,0)&=\delta_n,\\
G(n,t)&=G(n+N,t).
\end{aligned}\right.
\end{equation}

Given this Green's function, the solution of~\eqref{eq:reducemodel} is given by
\[
y(n,t)=\int_0^t\Lr{\sum_{m=1}^NG(n-m,t-s)f(m)}\md s,
\]
where $f$ is given by~\eqref{eq: gf} under the transform~\eqref{eq:number}.
As a result, we have
\[
\sum_{m=1}^NG(n-m,t)f(m)=\dfrac{\ka_2}{\eps}
\Lr{2G(n-L,t)-G(n-L-1,t)-G(n-L+1,t)},
\]
where we have set $L=N/2$.

By separation of variables, we have the following explicit form for
the lattice Green's function $G$:
\begin{equation}\label{eq:latgreen'}
G(n,t)=\dfrac{t}{N}+\dfrac1N\sum_{k=1}^{N-1}\dfrac{\sin[\om_k t]}{\om_k}
\cos\dfrac{2kn\pi}N,
\end{equation}
with \(\omega_k\) being the dispersion relation given by
\(
\om_k=(2/\eps)\sqrt{\ka_1+4\ka_2}\sin(k\pi/N).
\)

Using $\sum_{m=1}^Nf(m)=0$, we write
\[
\sum_{m=1}^NG(n-m,t-s)f(m)=\dfrac{4\ka_2}{N\eps}
\sum_{k=1}^{N-1}\dfrac{\sin[\om_k(t-s)]}{\om_k}
\sin^2\dfrac{k\pi}N\cos\dfrac{2k\pi}N(n-L).
\]
This leads to
\[
y(n,t)=\int_0^t\Lr{\sum_{m=1}^NG(n-m,t-s)f(m)}\md s
=\dfrac{\eps}{N}\dfrac{2\ka_2}{\ka_1+4\ka_2}
\sum_{k=1}^{N-1}\sin^2\dfrac{\om_kt}2\cos\dfrac{2k\pi}N(n-L).
\]
Using the above expression we bound $y(n,t)$ as
\begin{equation}\label{eq:point-solu}
\abs{y(n,t)}\le\dfrac{2\abs{\ka_2}}{\ka_1+4\ka_2}\eps.
\end{equation}
This estimate shows that the magnitude of the error
$y(n,t)$ is as small as $\mc{O}(\eps)$ for all $n$ and all time
$t$. This in turn
suggests that the error  induced by the ghost force is small in the maximum norm, which 
is consistent with that of the static problem~\cite{DobsonLuskin:2009a, MingYang:2009,
ChenMing:2012}.

Next we consider the discrete gradient of the error. A direct calculation gives
\begin{align*}
D y(n,t)&\equiv\dfrac{y(n+1,t)-y(n,t)}\eps\\
&=-\dfrac{1}N\dfrac{4\ka_2}{\ka_1+4\ka_2}\sum_{k=1}^{N-1}
\sin\dfrac{2k\pi}N(n+1/2-L)\sin\dfrac{k\pi}N\sin^2\dfrac{\om_kt}2.
\end{align*}
Clearly we may write the above expression as
\begin{equation}\label{eq:derexp}
D y(n,t)=-\dfrac{1}N\dfrac{4\ka_2}{\ka_1+4\ka_2}\sum_{k=0}^N
\sin\dfrac{2k\pi}N(n+1/2-L)\sin\dfrac{k\pi}N\sin^2\dfrac{\om_kt}2.
\end{equation}
 It follows from the above expression that $D y(n,t)$ is anti-symmetric in the 
sense of
\[
D y(n,t)=-D y(N-n-1).
\]
Therefore, we only consider the case $n\ge L$.
By~\eqref{eq:point-solu} we bound $D y(n,t)$ trivially:
\[
\abs{D y(n,t)}\le\dfrac{4\abs{\ka_2}}{\ka_1+4\ka_2}.
\]
This shows that
$D y(n,t)$ is uniformly bounded for all $n$ and all time $t$.
In the next two sections, we seek for a refined pointwise estimate of
$D y(n,t)$ when $t$ is of $\mc{O}(1)$ and of $\mc{O}(\eps)$. 
Notice that the same method can be employed to obtain
a refined pointwise estimate of $y(n,t)$. We leave it to the interested
readers. 
\section{Estimate of the error over long time}
In this section, we estimate the error for \(t=\mathcal{O}(1)\). By~\eqref{eq:derexp}, we write
$D y(L,t)$ as
\[
D y(L,t)=-\dfrac{1}N\dfrac{4\ka_2}{\ka_1+4\ka_2}\sum_{k=0}^N
\sin^2\dfrac{k\pi}N\sin^2\dfrac{\om_kt}2.
\]
Using the identity
\[
\sum_{k=0}^N\sin^2\dfrac{k\pi}N=\dfrac{N}2.
\]
We write
\begin{align}
D y(L,t)&=-\dfrac{1}N\dfrac{2\ka_2}{\ka_1+4\ka_2}\sum_{k=0}^N
\sin^2\dfrac{k\pi}N+\dfrac{1}N\dfrac{2\ka_2}{\ka_1+4\ka_2}
\sum_{k=0}^N\sin^2\dfrac{k\pi}N\cos(\om_kt)\nn\\
&=-\dfrac{\ka_2}{\ka_1+4\ka_2}+\dfrac{1}N\dfrac{2\ka_2}{\ka_1+4\ka_2}
\sum_{k=0}^N\sin^2\dfrac{k\pi}N\cos(\om_kt),
\label{eq:midexp}
\end{align}

When \(n\ne L\), we use the fact that
\[
\sum_{k=0}^N
\sin\dfrac{2k\pi}N(n+1/2-L)\sin\dfrac{k\pi}N=0,
\]
and we write the expression of $D y(n,t)$ in~\eqref{eq:derexp} as
\[
D y(n,t)=\dfrac{1}N\dfrac{2\ka_2}{\ka_1+4\ka_2}\sum_{k=0}^{N}
\sin\dfrac{2k\pi}N(n+1/2-L)\sin\dfrac{k\pi}N\cos(\om_kt),
\]
which can be further decomposed into
\begin{align}
Dy(n,t)&=\dfrac1{N}\dfrac{\ka_2}{\ka_1+4\ka_2}\sum_{k=0}^N
\sin\dfrac{k\pi}N\biggl\{\sin\Lr{\om_kt+\dfrac{2k\pi}N(n+1/2-L)}\nn\\
&\phantom{\dfrac1{N}\dfrac{\ka_2}{\ka_1+4\ka_2}\sum_{k=0}^N
\sin\dfrac{k\pi}N}\qquad
-\sin\Lr{\om_kt-\dfrac{2k\pi}N(n+1/2-L)}\biggr\}.\label{eq:genexp}
\end{align}

To bound $D y(n,t)$, we need to estimate an exponential
sum of the following form,
\[
\sum_{k=0}^Ne(f(k)),
\]
where the shorthand $e(f(k))\equiv\exp(2\pi\I f(k))$ is assumed.
The basic tool that will be used is the truncated form of the Poisson summation formula
 due to { Van der Corput}~\cite {VanderCorput:1921}.
The following form
with an explicit estimate for the remainder term can be found in~\cite [Lemma 7]{KaratsubaKorolev:2007}.
%
\begin{theorem}{\sc (Truncated Poisson)}~\label {thm:TruncPoisson}
Let $f$ and $\phi$ be real-valued functions satisfying the following conditions
on a closed interval $[a,b]$:
\begin{enumerate}
\item  $f^{\prime\prime}$ and $\phi^\prime(x)$ are continuous;
\item $0<f^{\prime\prime}(x)\le C_0$;
\item There are positive constants $H,U,\phi_0,\phi_1,\lambda$
such that $U\ge 1,0<b-a\le\lambda U$ and
\[
\abs{\phi(x)}\le\phi_0 H,\quad \abs{\phi^\prime(x)}\le\phi_1 H/U.
\]
\end{enumerate}
For any $\Delta,0<\Delta<1$, the equation
\begin{equation}\label{eq:tpoisson}
\sum_{a<n\le b}\phi(n)e(f(n))=\sum_{\al-\Delta\le\nu\le\beta+\Delta}
\int_a^b\phi(x)e(f(x)-\nu x)\dx+\theta R
\end{equation}
holds, where $\al=f'(a)$, $\beta=f'(b)$ and
\begin{align*}
R&=(\phi_0+\lambda\phi_1)H\Bigl(9.42+9C_0+12\Delta+\pi^{-1}\bigl(10\Delta^{-1}
+2\ln\Delta^{-1}\\
&\phantom{=(\phi_0+\lambda\phi_1)H}\quad+4.5(1+\Delta)^{-1}-4.5\ln (1+\Delta)+6.5\ln(\beta-\al+2)\bigr)\Bigr).
\end{align*}
Here $\theta$ is a function such that $\abs{\theta}\le 1$.
\end{theorem}

The assumption $f^{\prime\prime}>0$ can be relaxed to either
$f^{\prime\prime}\ge 0$ or $f^{\prime\prime}\le 0$. In the latter case, the
second condition is replaced by $-C_0\le f^{\prime\prime}(x)\le 0$.
\subsection{The estimate for $D y(N/2,t)$}
To bound $D y(N/2,t)$, we start with~\eqref{eq:midexp}. Based on the above theorem, we transform the exponential sum in~\eqref{eq:midexp} to a shorter sum with a bounded remainder. To clarify the dependance of the constant, we denote
\[
\ga=t\sqrt{\ka_1+4\ka_2},
\]
and assume that $1\le\ga\le N$ since $t=\mc{O}(1)$ and $N\ge 2$. We also
denote by $\integer{p}$ the integer part of a real number $p$, and
denote its fractional part by $\fraction{p}=p-\integer{p}$.
\begin{lemma}~\label{lema:trans}
Let $\phi(x)=\sin^2(\pi x/N)$ and $f(x)=\ga/(\pi\eps)\sin(\pi x/N)$.
There holds
\begin{align}\label{eq:fexpsum}
\dfrac{1}N\labs{\sum_{k=0}^N\phi(k)e(f(k))}
&\le\labs{\dfrac2N
\sum_{\al-1/2\le\nu\le\ga+1/2}\int_0^{N/2}\phi(x)
e(f(x)-\nu x)\dx}\nn\\
&\quad+C\eps(1+\eps\ga+\log(\ga+2)),
\end{align}
where $C$ is independent of $N, t$ and $\ga$.
\end{lemma}

\begin{proof}
It is easy to write the exponential sum into
\[
\sum_{k=0}^N\phi(k)e(f(k))=
2\sum_{k=0}^{N/2}\phi(k)e(f(k))-\phi(N/2)e(f(N/2)).
\]
With such choice of $f$ and $\phi$, we have
\begin{align*}
a&=0,b=N/2,\al=0,\beta=\ga,\\
\phi_0&=\phi_1=1,H=\pi,\lam=1,U=N/2.
\end{align*}
Setting \(\Del=1/2\) and using Theorem~\ref {thm:TruncPoisson}, we obtain
\[
\dfrac1N\sum_{k=0}^{N/2}\phi(k)e(f(k))=\dfrac1N
\sum_{\al-1/2\le\nu\le\ga+1/2}\int_0^{N/2}\phi(x)
e(f(x)-\nu x)\dx+\theta\dfrac{R}N,
\]
where
\[
R=2\pi\Bigl(9.42+9\pi\ga/N+6+\pi^{-1}\bigl(20
+2\ln2+3-4.5\ln (3/2)+6.5\ln(\ga+2)\bigr)\Bigr).
\]
This immediately implies that there exists a constant $C$ such that
\[
\abs{\theta R/N}\le C\eps\Lr{1+\eps\ga+\ln(\ga+2)}.
\]
We obtain~\eqref{eq:fexpsum} by combining the above two inequalities.
\end{proof}
\begin{remark}
The choice of $\Delta$ is not unique. However, it cannot be too small. Otherwise,
the remainder term blows up as $\Delta\to 0$.
\end{remark}

To bound the shorter sum in~\eqref{eq:fexpsum}, we shall
rely on the following first derivative test.
\begin{lemma}[First derivative test]\cite [Lemma 1, p. 47]
{Montgomery:1994}\label{lemma:firsttest}
Let $r(x)$ and $\theta(x)$ be real-valued functions on $[a,b]$ such that
$r(x)$ and $\theta'(x)$ are continuous. Suppose that $\theta'(x)/r(x)$
is positive and monotonically increasing in this interval. If $0<\lam_1\le\theta'(a)/r(a)$,
then
\[
\labs{\int_a^b r(x)e(\theta(x))\dx}\le\dfrac{1}{\pi\lam_1}.
\]
\end{lemma}
\begin{remark}
If $\theta'(x)/r(x)$ is negative and monotonically decreasing on $[a,b]$ and
\[
\theta'(a)/r(a)\le\lam_1<0,
\]
then we obtain the same bound by taking complex conjugates. Moreover, if $\theta'(x)/r(x)$
is monotone on $[a,b]$ and
\[
\abs{\theta'(x)/r(x)}\ge\lam_1>0\qquad x\in (a,b),
\]
we obtain the same bound by combining the above two cases.
\end{remark}

We write
\[
\dfrac1N\int_0^{N/2}\phi(x)
e(f(x)-\nu x)\dx=\dfrac1{\pi}\int_0^{\pi/2}\vphi(y)e\Lr{G_\nu(y)}\dy,
\]
where $\vphi(y)=\sin^2 y$, and
$G_{\nu}(y)=(N/\pi)(\ga\sin y-\nu y)$ for any $\nu\in\mb{N}$. We define
$F_\nu(y)=G_\nu'(y)/\vphi(y)$. By Lemma~\ref{lema:trans}, it remains to
estimate the integral
\(
\int_0^{\pi/2}\vphi(y)e(G_{\nu}(y))\dy
\)
for $\nu=0,\cdots,\integer{\ga+1/2}$. The three cases
$\nu=0, \nu=1,\cdots, \integer{\ga+1/2}-1$
and $\nu=\integer{\ga+1/2}$ will be treated separately in the following lemmas.
\begin{lemma}\label{lemma:zeroOrder}
There holds
\begin{equation}\label{eq:zeroOrder}
\labs{\int_0^{\pi/2}
\vphi(y)e\Lr{G_0(y)}\dy}\le 2(N\ga)^{-1/2}\;.
\end{equation}
\end{lemma}

\begin{proof}
For any $\del\in (0,\pi/2)$ that
will be determined later on, we have, for any $y\in (0,\pi/2-\del)$,
\[
F_0(y)\ge F_0(\pi/2-\del)
=\dfrac{N\ga}\pi\dfrac{\sin\del}{\cos^2\del}
\ge\dfrac{N\ga}{\pi}\tan\del\ge\dfrac{N\ga\del}{\pi},
\]
where we have used the fact that
\(\tan x\ge x\) for \(x\in [0,\pi/2]\). Using Lemma~\ref{lemma:firsttest} with $\lam_1=N\ga\del/\pi$, we obtain
\[
\labs{\int_0^{\pi/2-\del}
\vphi(y)e\Lr{G_0(y)}\dy}\le\dfrac{1}{N\ga\del}.
\]
The integral over the complementary portion of the interval can be bounded trivially:
\[
\labs{\int_{\pi/2-\del}^{\pi/2}\vphi(y)e\Lr{G_0(y)}\dy}\le
\int_{\pi/2-\del}^{\pi/2}\dy\le\del.
\]
On adding the two estimates we deduce that
\[
\labs{\int_0^{\pi/2}
\vphi(y)e\Lr{G_0(y)}\dy}\le\dfrac{1}{N\ga\del}+\del.
\]
This is minimized by taking $\del=(N\ga)^{-1/2}\in(0,\pi/2)$, and we obtain~\eqref{eq:zeroOrder}.
\end{proof}

The second case is more involved since $F_\nu$ changes sign over $(0,\pi/2)$.
\begin{lemma}\label{lema:genest1}
If $1\le\nu<\ga$, then
\begin{equation}\label{eq:genest1}
\labs{\int_0^{\pi/2}\vphi(y)e(G_{\nu}(y))\dy}\le
\dfrac{3\pi}{\sqrt{N\ga}}.
\end{equation}
\end{lemma}

\begin{proof}
For $1\le\nu<\ga$, there exists $y_{\nu}\in(0,\pi/2)$ such that
$F_{\nu}(y_{\nu})=0$ with $\cos y_{\nu}=\nu/\ga$. For any
$\eta\in \Lr{0,\min(y_\nu,\pi/2-y_\nu)}$
that will be chosen later, we write
\[
\int_0^{\pi/2}\vphi(y)e\Lr{G_{\nu}(y)}\dy=\Lr{\int_0^{y_{\nu}-\eta}
+\int_{y_{\nu}+\eta}^{\pi/2}
+\int_{y_{\nu}-\eta}^{y_\nu+\eta}}
\vphi(y)e\Lr{G_{\nu}(y)}\dy.
\]
We deal with the three integrals separately.

Using Lemma~\ref{lemma:firsttest}  with $\lam_1=\abs{F_{\nu}(y_{\nu}-\eta)}$,
we obtain
\[
\labs{\int_0^{y_{\nu}-\eta}\vphi(y)e(G_\nu(y))\dy}
\le\dfrac{1}{\pi\abs{F_{\nu}(y_{\nu}-\eta)}},
\]
and
\begin{align*}
\abs{F_{\nu}(y_{\nu}-\eta)}&=F_{\nu}(y_{\nu}-\eta)
=\dfrac{N\ga}{\pi}\dfrac{\cos(y_{\nu}-\eta)-\cos y_{\nu}}{\sin^2(y_{\nu}-\eta)}\\
&=\dfrac{2N\ga}{\pi}\dfrac{\sin(y_{\nu}-\eta/2)\sin(\eta/2)}{\sin^2(y_{\nu}-\eta)}\\
&\ge\dfrac{2N\ga}{\pi}\dfrac{\sin(\eta/2)}{\sin y_{\nu}}\\
&\ge\dfrac{2N\ga\eta}{\pi^2\sin y_{\nu}},
\end{align*}
where we have used Jordan's inequality
\begin{equation}\label{eq:jordan}
\sin x\ge \dfrac2{\pi}x\qquad x\in [0,\pi/2].
\end{equation}
This gives
\[
\labs{\int_0^{y_{\nu}-\eta}
\vphi(y)e\Lr{G_{\nu}(y)}\dy}\le\dfrac{\pi\sin y_{\nu}}{2N\ga\eta}.
\]
Using Lemma~\ref{lemma:firsttest} again
with $\lam_1=\abs{F_{\nu}(y_{\nu}+\eta)}$, we have, for the second integral,
\[
\labs{\int_{y_{\nu}+\eta}^{\pi/2}\vphi(y)e(G_\nu(y))\dy}
\le\dfrac{1}{\pi\abs{F_{\nu}(y_{\nu}+\eta)}}.
\]
Furthermore,
\begin{align*}
\abs{F_{\nu}(y_{\nu}+\eta)}&=-F_{\nu}(y_{\nu}+\eta)
=\dfrac{N\ga}{\pi}\dfrac{\cos y_{\nu}-\cos(y_{\nu}+\eta)}{\sin^2(y_{\nu}+\eta)}\\
&=\dfrac{2N\ga}{\pi}\dfrac{\sin(y_{\nu}+\eta/2)\sin(\eta/2)}{\sin^2(y_{\nu}+\eta)}\\
&\ge\dfrac{2N\ga\eta}{\pi^2}\dfrac{\sin(y_{\nu}+\eta/2)}{\sin^2(y_{\nu}+\eta)}.
\end{align*}
This leads to
\[
\labs{\int_{y_{\nu}+\eta}^{\pi/2}
\vphi(y)e\Lr{G_{\nu}(y)}\dy}
\le\dfrac{\pi}{2N\ga\eta}\dfrac{\sin^2(y_{\nu}+\eta)}{\sin(y_{\nu}+\eta/2)}.
\]
If \(y_\nu \in (0, \pi/4]\), we would require that \(\eta \in (0, y_\nu)\).
We will have, \(\sin (y_\nu + \eta)\le \sin 2y_\nu\le 2 \sin y_\nu\)  and
\(\sin y_\nu<\sin(y_\nu+\eta/2)\) since \(y_\nu<y_\nu+\eta/2<2 y_\nu\le\pi/2\). The bound for the above integral is changed to
\[
\labs{\int_{y_{\nu}+\eta}^{\pi/2}
\vphi(y)e\Lr{G_{\nu}(y)}\dy}\le \dfrac{2\pi\sin y_\nu}{N\ga\eta}.
\]
We estimate the remaining integral trivially:
\[
\labs{\int_{y_{\nu}-\eta}^{y_\nu+\eta}
\vphi(y)e\Lr{G_{\nu}(y)}\dy}\le 2\eta.
\]
Summing up all the above estimates, we obtain
\[
\labs{\int_0^{\pi/2}
\vphi(y)e\Lr{G_{\nu}(y)}\dy}\le\dfrac{5\pi\sin y_\nu}{2N\ga\eta}+2\eta.
\]
Taking $\eta=(N\ga)^{-1/2}\sin y_{\nu}$, which is less than $y_\nu$, we obtain
\[
\labs{\int_0^{\pi/2}
\vphi(y)e\Lr{G_{\nu}(y)}\dy}\le \dfrac{5\pi}{2\sqrt{N \ga} }
+\dfrac{2\sin y_\nu}{\sqrt{N \ga} }\le\dfrac{5\pi/2+\sqrt2}{\sqrt{N\ga}}<\dfrac{3\pi}{\sqrt{N \ga} } .
\]

On the other hand, if \(y_\nu \in (\pi/4,\pi/2]\), we would require that \(\eta\in (0, \pi/2-y_\nu)\).  We will have $\sin y_\nu<\sin (y_\nu+\eta/2)$ since
$y_\nu<y_\nu+\eta/2<y_\nu+\eta\le\pi/2$, and \(
\sin^2(y_\nu+\eta)\le 1\le 2\sin^2 y_\nu\)
since \(\sin^2 y_\nu \ge 1/2\). We bound
the second integral as
\[
\labs{\int_{y_{\nu}+\eta}^{\pi/2}
\vphi(y)e\Lr{G_{\nu}(y)}\dy}\le \dfrac{\pi\sin y_\nu}{N\ga \eta}.
\]
This yields
\[
\labs{\int_0^{\pi/2}
\vphi(y)e\Lr{G_{\nu}(y)}\dy}\le\dfrac{3\pi\sin y_\nu}{2N\ga\eta}+2\eta.
\]
In this case, we can choose \(\eta=(N\ga)^{-1/2}\)  provided that
\begin{equation}\label{eq:cond0}
\eta<\pi/2-y_\nu.
\end{equation}
This immediately implies that
\[
\labs{\int_0^{\pi/2}
\vphi(y)e\Lr{G_{\nu}(y)}\dy}\le\dfrac{3\pi}{\sqrt{N\ga}}.
\]
Such choice of $\eta$ is feasible since
\[
\nu\ge 1>\sqrt{\ga/N},
\]
which yields \(\nu>\ga\eta\), or equivalently, \(\eta<\cos y_\nu=\sin(\pi/2-y_\nu)
\le\pi/2-y_\nu\). This
directly gives~\eqref{eq:cond0}. Finally we get~\eqref{eq:genest1}.
\end{proof}

Next we consider the endpoint case $\nu=\integer{\ga+1/2}$.
\begin{lemma}\label{lema:gen}
Let $\nu=\integer{\ga+1/2}$. If $\nu\ge\ga$, then
\begin{equation}\label{eq:maxOrder}
\labs{\int_0^{\pi/2}\vphi(y)e(G_{\nu}(y))\dy}\le\dfrac4{N\ga}.
\end{equation}
If $\nu<\ga$, then
\begin{equation}\label{eq:maxOrder2}
\labs{\int_0^{\pi/2}\vphi(y)e\Lr{G_{\nu}(y)}\dy}\le
\dfrac{3\pi}{\sqrt{N\ga}}.
\end{equation}
\end{lemma}

\begin{proof}
If $\nu\ge\ga$, then there exists a stationary point $y_\nu$ of $F_\nu(y)$ with
$\cos y_\nu=\nu/\ga-\sqrt{\nu^2/\ga^2-1}$. Because
$F_\nu(y)$ is monotonically increasing over
$(0,y_\nu)$ and monotonically decreasing over $(y_\nu,\pi/2)$, we get
$\min_{y\in(0,\pi/2)}\abs{F_{\nu}(y)}\ge\abs{F_{\nu}(y_\nu)}$.
To each of the two intervals we apply
Lemma~\ref{lemma:firsttest} with $\lam=\abs{F_\nu(y_\nu)}$.
On adding these estimates we deduce that
\begin{equation}\label{eq:subinter}
\labs{\int_0^{\pi/2}\vphi(y)e(G_{\nu}(y))\dy}\le\dfrac{2}{\pi\abs{F_{\nu}(y_\nu)}},
\end{equation}
which yields~\eqref{eq:maxOrder} by using
\[
\abs{F_{\nu}(y_\nu)}
=\dfrac{N\ga}{2\pi}\Lr{\nu/\ga+\sqrt{\nu^2/\ga^2-1}}
\ge\dfrac{N\ga}{2\pi}.
\]

When $\nu=\integer{\ga+1/2}<\ga$, we invoke the
estimate~\eqref{eq:genest1} to get~\eqref{eq:maxOrder2}.
\end{proof}

Summing up the above estimates for the shorter sum, we obtain
the estimate for the exponentinal sum in~\eqref{eq:fexpsum}.
\begin{lemma}~\label{lem:expsum}
There holds
\begin{equation}\label{eq:fexpsum1}
\dfrac{1}N\labs{\sum_{k=0}^N\phi(k)e(f(k))}\le C\Lr{\eps(1+\eps\ga+\log(\ga+2))+\dfrac{\eps}\ga+
\sqrt{\eps\ga}},
\end{equation}
where $C$ is independent of $N, t$ and $\ga$.
\end{lemma}

\begin{proof}
Denote $\nu_0=\integer{\ga+1/2}$. If $\nu_0\ge\ga$, then
\begin{align*}
\sum_{\nu=0}^{\nu_0}\int_0^{\pi/2}
\vphi(y)e\Lr{G_{\nu}(y)}\dy&=\int_0^{\pi/2}
\vphi(y)e\Lr{G_0(y)}\dy+\sum_{\nu=1}^{\nu_0-1}\int_0^{\pi/2}
\vphi(y)e\Lr{G_{\nu}(y)}\dy\\
&\quad
+\int_0^{\pi/2}
\vphi(y)e\Lr{G_{\nu_0}(y)}\dy.
\end{align*}
Using the estimates~\eqref{eq:zeroOrder},~\eqref{eq:genest1}, and~\eqref{eq:maxOrder},
we bound the right hand side of the above sum as follows,
\begin{align}
\labs{\sum_{\nu=0}^{\nu_0}\int_0^{\pi/2}
\vphi(y)e\Lr{G_{\nu}(y)}\dy}&\le 2\Lr{N\ga}^{-1/2}
+3\pi\sum_{\nu=1}^{\nu_0-1}(N\ga)^{-1/2}+4(N\ga)^{-1}\nn\\
&\le\pi(N\ga)^{-1/2}(3\ga+1)+4(N\ga)^{-1}.\label{eq:caseone}
\end{align}

If $\nu_0<\ga$, we have $\nu_0=\integer{\ga}$, and
\begin{align*}
\sum_{\nu=0}^{\nu_0}\int_0^{\pi/2}
\vphi(y)e\Lr{G_{\nu}(y)}\dy&=\int_0^{\pi/2}
\vphi(y)e\Lr{G_0(y)}\dy+\sum_{\nu=1}^{\nu_0}\int_0^{\pi/2}
\vphi(y)e\Lr{G_{\nu}(y)}\dy.
\end{align*}
Proceeding along the same line that leads to~\eqref{eq:caseone}, we obtain
\begin{equation}
\labs{\sum_{\nu=0}^{\nu_0}\int_0^{\pi/2}
\vphi(y)e\Lr{G_{\nu}(y)}\dy}\le  2(N\ga)^{-1/2}+3\pi(\ga/N)^{1/2}.\label{eq:casetwo}
\end{equation}
Combining the estimates~\eqref{eq:caseone},~\eqref{eq:casetwo}
and~\eqref{eq:fexpsum}, we obtain~\eqref{eq:fexpsum1}.
\end{proof}

Substituting the estimates~\eqref{eq:fexpsum1} and~\eqref{eq:fexpsum}
into~\eqref{eq:midexp}, we obtain
the pointwise estimate for $D(N/2,t)$ as follows.
\begin{theorem}\label{thm:main1}
If $t=\mc{O}(1)$ and  $n=N/2$ or $N/2-1$, then
\begin{equation}\label{eq:estmax}
\labs{D y(n,t)+\dfrac{\ka_2}{\ka_1+4\ka_2}}\le C\dfrac{\abs{\ka_2}}{\ka_1+4\ka_2}
\biggl(\eps[1+\eps\ga+\log(\ga+2)]
+\sqrt{\dfrac{\eps}\ga}(1+\ga)\biggr),
\end{equation}
where $C$ is independent of $N,t$ and $\ga$.
\end{theorem}

The above estimate means that $Dy(n,t)$ is in the $\mc{O}(\sqrt\eps)-$neighborhood
of \allowbreak $\ka_2/(\ka_1+4\ka_2)$ when $n=N/2$ or $n=N/2-1$.
\subsection{The estimate for $Dy (n,t)$ with $n\ne L$}
By~\eqref{eq:genexp}, we need to estimate
two exponential sums $\sum_{k=0}^N\phi(k)e(f(k))$ with
\[
\phi(x)=\sin\dfrac{\pi x}N,\quad f(x)=\dfrac{\ga}{\pi\eps}\sin\dfrac{\pi x}N+
\dfrac{n+1/2-L}N x,
\]
and
\[
\phi(x)=\sin\dfrac{\pi x}N, \quad f(x)=\dfrac{\ga}{\pi\eps}\sin\dfrac{\pi x}N-
\dfrac{n+1/2-L}N x.
\]

In what follows we only give the full details for estimating of
the first exponential sum, and the same proof works for
the second exponential sum. Denote by $\varrho=(n+1/2-L)/N$,
proceeding along the same line that leads to~\eqref{eq:fexpsum} and choosing $\Delta=\max(1/2,1-\fraction{\ga+\varrho})$, we get
\begin{align}\label{eq:gexpsum}
\dfrac{1}N\labs{\sum_{k=0}^N\phi(k)e(f(k))}
&\le\labs{\dfrac2{\pi}
\sum_{\nu=0}^{\integer{\ga+\varrho}+1}\int_0^{\pi/2}\vphi(y)
e(G_\nu(y))\dy}\nn\\
&\quad+C\eps(1+\eps\ga+\log(\ga+2)),
\end{align}
where $C$ is independent of $N, t$ and $\ga$. Here $\vphi(y)=\sin y$ and
\(
G_{\nu}(y)=(N/\pi)(\ga\sin y+\varrho y-\nu y).
\)
We also define $F_\nu(y)=G_\nu^\prime(y)/\vphi(y)$.
\begin{lemma}
There holds
\begin{equation}\label{eq:genzeroOrder}
\labs{\int_0^{\pi/2}\vphi(y)e(G_0(y))\dy}\le\min\Lr{2(N\ga)^{-1/2},
\dfrac{1}{n+1/2-L}}.
\end{equation}
\end{lemma}

\begin{proof}
If $2(N\ga)^{-1/2}<1/(n+1/2-L)$, then we proceed along the same line that leads to Lemma~\ref{lemma:zeroOrder} to obtain
\[
\labs{\int_0^{\pi/2}\vphi(y)e(G_0(y))\dy}\le 2(N\ga)^{-1/2}.
\]
Otherwise, using Lemma~\ref{lemma:firsttest} with $\lam=(n+1/2-L)/\pi$, we obtain
\[
\labs{\int_0^{\pi/2}\vphi(y)e(G_0(y))\dy}\le
\dfrac{1}{n+1/2-L}.
\]
Combining the above two inequalities gives~\eqref{eq:genzeroOrder}.
\end{proof}
\begin{lemma}\label{lemma:maxorder}
Let $\nu=\integer{\ga+\varrho}+1$, then
\begin{equation}\label{eq:maxorder1'}
\labs{\int_0^{\pi/2}\vphi(y)e(G_{\nu}(y))\dy}
\le 4(N\ga)^{-2/3}.
\end{equation}
\end{lemma}

\begin{proof}
The function $F_\nu$ has a stationary point $y_\nu$ that
satisfies $\cos y_\nu=\ga/(\nu-\varrho)$. In this case,
applying the {\em first derivative test} directly to the integral may yield
a bound of the form \(1/(N\sin y_\nu)\), which is undesirable since \(y_\nu\)
can be very close to zero when \(\nu\) is close to \(\ga+\varrho\).
Instead, for any $\del\in(0,\pi/2)$ to be determined later on, we have
\[
\labs{\int_0^{\del}\vphi(y)e(G_\nu(y))\dy}\le
\int_0^{\del}\vphi(y)\dy\le\int_0^{\del}y\dy=\dfrac{\del^2}2.
\]

If $y_\nu\le\del$,
then we use Lemma~\ref{lemma:firsttest} with $\lam=\abs{F_\nu(\del)}$. This gives
\[
\labs{\int_{\del}^{\pi/2}\vphi(y)e(G_\nu(y))\dy}\le
\dfrac{1}{\pi\abs{F_\nu(\del)}}.
\]

If $y_\nu>\del$, then we proceed along the same line that leads to~\eqref{eq:subinter} to get
\[
\labs{\int_{\del}^{\pi/2}\vphi(y)e(G_\nu(y))\dy}\le\dfrac{2}{\pi\min_{y\in[\del,\pi/2]}
\abs{F_\nu(y)}}=\dfrac{2}{\pi\abs{F_\nu(y_\nu)}}.
\]
A direct calculation gives
\begin{align*}
\abs{F_\nu(\del)}&=\dfrac{N}{\pi}(\nu-\varrho)\dfrac{1-\cos y_{\nu}\cos\del}{\sin\del}\ge\dfrac{N}{\pi}(\nu-\varrho)\dfrac{1-\cos\del}{\sin\del}\\
&=\dfrac{N}{\pi}(\nu-\varrho)\tan\dfrac{\del}2\ge\dfrac{N(\nu-\varrho)}{2\pi}\del,
\end{align*}
and for $y_\nu>\del$,
\begin{equation}\label{eq:lowbd}
\abs{F_\nu(y_\nu)}
=\dfrac{N}{\pi}(\nu-\varrho)\sin y_\nu\ge\dfrac{2N(\nu-\varrho)}{\pi^2}y_\nu
>\dfrac{2N(\nu-\varrho)}{\pi^2}\del.
\end{equation}
Combining the above four inequalities, we obtain,
for any $\del\in(0,\pi/2)$,
\[
\labs{\int_{\del}^{\pi/2}\vphi(y)e(G_\nu(y))\dy}\le\max\Lr{\dfrac1{\pi\abs{F_\nu(\del)}},
\dfrac2{\pi\abs{F_\nu(y_\nu)}}}\le\dfrac{\pi}{N(\nu-\varrho)\del}.
\]
To sum up, we have
\[
\labs{\int_0^{\pi/2}\vphi(y)e(G_\nu(y))\dy}\le\dfrac{\del^2}2
+\dfrac{\pi}{N(\nu-\varrho)\del}.
\]
Taking $\del=\pi^{1/3}(N(\nu-\varrho))^{-1/3}\in(0,\pi/2)$, we get
\begin{equation}\label{eq:extremeest}
\labs{\int_0^{\pi/2}\vphi(y)e(G_\nu(y))\dy}\le\dfrac32\pi^{2/3}[N(\nu-\varrho)]^{-2/3}
\le 4[N(\nu-\varrho)]^{-2/3},
\end{equation}
which yields~\eqref{eq:maxorder1'} by using the fact that
$\nu-\varrho\ge\ga+\varrho-\varrho=\ga$.
\end{proof}

Proceeding along the same line that led to~\eqref {eq:maxOrder2}, we obtain
a parallel result of Lemma~\ref{lema:genest1}. The proof is postponed to the
appendix.
\begin{lemma}\label{lema:altergen}
If $1\le\nu<\integer{\ga+\varrho}$, then
\begin{equation}\label{eq:genest}
\labs{\int_0^{\pi/2}\vphi(y)e(G_{\nu}(y))\dy}\le
\dfrac{3\pi}{\sqrt{N\ga}\sin y_\nu}.
\end{equation}
\end{lemma}

The estimate for the
endpoint case $\nu=\integer{\ga+\varrho}$
is essentially the same with the argument that led to~\eqref{eq:genest}. However,
the root $y_\nu$ varies with the magnitude of the fractional part
of $\ga+\varrho$. Therefore, a more careful treatment is required to obtain a bound
that is independent of the magnitude of $\fraction{\ga+\varrho}$.  The proof is also
postponed to the appendix.
\begin{lemma}\label{lemma:maxorder1}
If $\nu=\integer{\ga+\varrho}$, then
\begin{equation}\label{eq:maxorder1}
\labs{\int_0^{\pi/2}\vphi(y)e(G_{\nu}(y))\dy}
\le 4\pi(N\ga)^{-1/2}.
\end{equation}
\end{lemma}

Combining these lemmas, we have the following estimate.
\begin{theorem}\label{thm:main2}
If $t=\mc{O}(1)$, and  \(n\ne N/2,N/2-1\), then
\begin{equation}\label{eq:pointgen}
\abs{D y(n,t)}\le C\dfrac{\abs{\ka_2}}{\ka_1+4\ka_2}
\Biggl(\sqrt{\dfrac{\eps}\ga}(1+\ga)+\eps(1+\eps\ga
+\log(\ga+2))\Biggr),
\end{equation}
where $C$ is independent of $N,t$ and $\ga$.
\end{theorem}

\begin{proof}
Summing up the above three lemmas, we obtain
\begin{align*}
\labs{\sum_{\nu=0}^{\integer{\ga+\varrho}+1}\int_0^{\pi/2}\varphi(y)
e(G_\nu(y))\dy}&\le 2(N\ga)^{-1/2}+\sum_{\nu=1}^{\integer{\ga+\varrho}-1}
\labs{\int_0^{\pi/2}\varphi(y)
e(G_\nu(y))\dy}\\
&\quad+4\pi(N\ga)^{-1/2}+4(N\ga)^{-2/3}\\
&\le3\pi(N\ga)^{-1/2}
\sum_{\nu=1}^{\integer{\ga+\varrho}-1}
\dfrac{1}{\sin y_\nu}\\
&\quad+5\pi(N\ga)^{-1/2}+4(N\ga)^{-2/3}.
\end{align*}
A direct calculation gives
\begin{align*}
\sum_{\nu=1}^{\integer{\ga+\varrho}-1}
\dfrac{1}{\sin y_\nu}&=\sum_{\nu=1}^{\integer{\ga+\varrho}-1}
\dfrac{1}{\sqrt{1-(\nu-\varrho)^2/\ga^2}}\\
&\le\sum_{\nu=1}^{\integer{\ga+\varrho}-1}\int_{\nu}^{\nu+1}
\dfrac{1}{\sqrt{1-(x-\varrho)^2/\ga^2}}\dx\\
&=\int_1^{\integer{\ga+\varrho}}\dfrac{1}{\sqrt{1-(x-\varrho)^2/\ga^2}}\dx\\
&\le\int_{\varrho}^{\ga+\varrho}\dfrac{1}{\sqrt{1-(x-\varrho)^2/\ga^2}}\dx\\
&=\ga.
\end{align*}
Combining the above two inequalities, we obtain
\begin{equation}\label{eq:fsum1}
\labs{\sum_{\nu=0}^{\integer{\ga+\varrho}+1}\int_0^{\pi/2}\varphi(y)
e(G_\nu(y))\dy}\le
5\pi(N\ga)^{-1/2}(1+\ga)+4(N\ga)^{-2/3}.
\end{equation}

It remains to estimate the integral \(\int_0^{\pi/2}\vphi(y)e(G_{\nu}(y))\dy\)
with $\varphi(y)=\sin y$ and
\(G_{\nu}(y)=(N/\pi)(\ga\sin y-\varrho y-\nu y)\).
We choose $\Delta=1/2$. The remainder
is still bounded by $\mc{O}\Lr{\eps(1+\eps\ga+\log(\ga+2))}$.
Obviously, there holds
$-\varrho-\Delta>-1$ since $\abs{\varrho}\le 1/2$. It remains to
estimate the shorter sum with $\nu=0,\cdots,\integer{\ga-\varrho+1/2}$.
We deal with the cases when $\nu=\integer{\ga-\varrho}+1,
\nu=\integer{\ga-\varrho}$ and $\nu=0,\cdots,\integer{\ga-\varrho}-1$
exactly the same with those of Lemmas~\ref{lemma:maxorder}, ~\ref{lemma:maxorder1} and~\ref{lema:altergen}, respectively.
This results in
\[
\labs{\sum_{\nu=0}^{\integer{\ga-\varrho}+1}\int_0^{\pi/2}
\vphi(y)e(G_\nu(y))\dy}\le 5\pi(N\ga)^{-1/2}(1+\ga)+4(N\ga)^{-2/3},
\]
which together with~\eqref{eq:fsum1} gives the final estimate~\eqref{eq:pointgen}
\end{proof}

\medskip
\section{Estimate of the solution over short time}
In this section, we estimate the solution over a shorter time interval,
i.e., $t=\mc{O}(\eps)$. This is motivated by the previous observation that the error already develops to
finite magnitude within this short time scale.
\begin{lemma}
If $t=\mc{O}(\eps)$ and $n\not=N/2,N/2-1$, then
\begin{equation}\label{eq:midestshort}
\abs{Dy(n,t)}\le C\dfrac{\abs{\ka_2}}{\ka_1+4\ka_2}\Lr{
\abs{n+1/2-N/2}^{-2/3}+\eps^{2/3}}.
\end{equation}
where $C$ is independent of $N$ and $t$.
\end{lemma}

The proof of this lemma is essentially the same with that of
Theorem~\ref{thm:main2}.
\begin{proof}
As to $f(x)=\ga/(\pi\eps)\sin(\pi x/N)+
\varrho x$,  we have $\al=\varrho$ and $\beta=\ga+\varrho$.
We choose $\Del=1/2$, and the remainder term is bounded by $\mc{O}(\eps+\log(\ga+2)\eps)=\mc{O}(\eps)$.
There are only two terms in the shorter sum~\eqref{eq:tpoisson}, i.e.,
$\nu=0,1$ since $t=\mc{O}(\eps)$. When $\nu=0$, using Lemma~\ref{lemma:firsttest} with $\lam=(n+1/2-L)/\pi$, we obtain
\[
\labs{\int_0^{\pi/2}\vphi(y)e(G_0(y))\dy}\le\dfrac{1}{n+1/2-L}.
\]

Using~\eqref{eq:extremeest} with $\nu=1$,
we obtain
\[
\labs{\int_0^{\pi/2}\vphi(y)e(G_1(y))\dy}
\le 4[N-(n+1/2-L)]^{-2/3}
\le 8N^{-2/3}.
\]

As to $f(x)=\ga/(\pi\eps)\sin(\pi x/N)-
\varrho x$, we still take $\Del=1/2$, and the remainder is still bounded by $\mc{O}(\eps)$. There is only one term in the shorter sum, i.e., $\nu=0$. Proceeding along the same line
that leads to~\eqref{eq:extremeest}, we get, for any $\del\in (0,\pi/2)$,
\[
\labs{\int_0^{\pi/2}\vphi(y)e(G_0(y))\dy}\le\dfrac{\del^2}2
+\dfrac{\pi}{(n+1/2-L)\del},
\]
which is minimized by taking $\del=\pi^{1/3}(n+1/2-L)^{-1/3}\in (0,\pi/2)$. This gives
\[
\labs{\int_0^{\pi/2}\vphi(y)e(G_0(y))\dy}\le 4(n+1/2-L)^{-2/3}.
\]
Summing up the above estimates, we get~\eqref{eq:midestshort}.
\end{proof}

We use Euler-MacLaurin formula instead of the truncated Poisson summation formula
to bound $D y(N/2,t)$, because this approach gives a more explicit bound
for the remainder. The starting point is the
following first-derivative form of Euler-MacLaurin formula. For
any real valued function $f(x)$  in $[a,b]$ with continuous
first derivative, we have
\begin{align}\label{eq:euler-mac}
\int_a^b f(x)\dx&=\dfrac{b-a}{2N}\Lr{f(a)+f(b)}+\dfrac{b-a}N\sum_{k=1}^{N-1}
f\Lr{a+k\dfrac{b-a}N}\nn\\
&\quad-\dfrac{b-a}N\int_a^b\Lr{\dfrac{(x-a)N}{b-a}
-\left\lfloor\dfrac{(x-a)N}{b-a}\right\rfloor-\dfrac12}f'(x)\dx.
\end{align}

Starting with ~\eqref{eq:midexp}, and applying Euler-MacLaurin
formula~\eqref {eq:euler-mac} to
\[
f(x)=\sin^2x\cos[(2\ga/\eps)\sin x]
\]
with $a=0$ and $b=\pi$, we obtain
\[
\dfrac{1}N
\sum_{k=0}^N\sin^2\dfrac{k\pi}N\cos(\om_k t)=\negint_0^\pi f(x)\dx
+\dfrac{1}N\int_0^{\pi}\Lr{\dfrac{Nx}{\pi}
-\left\lfloor\dfrac{Nx}{\pi}\right\rfloor-\dfrac12}f'(x)\dx.
\]
The remainder can be directly bounded as
\[
\labs{\dfrac{1}N\int_0^{\pi}\Lr{\dfrac{Nx}{\pi}
-\left\lfloor\dfrac{Nx}{\pi}\right\rfloor-\dfrac12}f'(x)\dx}
\le\dfrac{\pi}N\Lr{2+2\ga/\eps}.
\]
The integral $\negint_0^\pi f(x)\dx$ can be calculated as follows.
\begin{align*}
\negint_0^\pi f(x)\dx&=\negint_0^{\pi/2}f(x)\dx=\negint_0^{\pi/2}\sin^2x\sum_{m=0}^\infty(-1)^m
\Lr{\dfrac{2\ga}\eps}^{2m}\dfrac{1}{(2m)!}\sin^{2m}x\dx\\
&=\sum_{m=0}^\infty(-1)^m
\Lr{\dfrac{2\ga}\eps}^{2m}\dfrac{1}{(2m)!}
\negint_0^{\pi/2}\sin^{2m+2}x\dx\\
&=\sum_{m=0}^\infty(-1)^m
\Lr{\dfrac{2\ga}\eps}^{2m}\dfrac{1}{(2m)!}
\dfrac{(2m+1)!!}{(2m+2)!!}\\
&=\sum_{m=0}^\infty(-1)^m
\Lr{\dfrac{\ga}\eps}^{2m}\dfrac{1}{(m!)^2}
\dfrac{2m+1}{2m+2}.
\end{align*}

This implies the following estimate for $D y(N/2,t)$ when $t=\mc{O}(\eps)$.
\begin{lemma}
If $t=\mc{O}(\eps)$ and $n=N/2$ or $n=N/2-1$, then
\begin{align}\label{eq:midestshort'}
&\labs{Dy(n,t)+\dfrac{\ka_2}{\ka_1+4\ka_2}-\dfrac{2\ka_2}{\ka_1+4\ka_2}
\sum_{m=0}^\infty(-1)^m
\Lr{\dfrac{\ga}\eps}^{2m}\dfrac{1}{(m!)^2}
\dfrac{2m+1}{2m+2}}\nn\\
&\hspace{1in}\le 2\pi(\eps+\ga).
\end{align}
\end{lemma}

The above estimate means that $Dy(N/2,t)$ is in an $\mc{O}(\eps)-$neighborhood of a constant value that depends on the ratio $t/\eps$.
\begin{remark}
We cannot directly use the above
approach to estimate $D y(n,t)$ because an undesirable term $\mc{O}(n/N)$ may
appear in the bound.
\end{remark}
\section{Discussion}
Based on a simple one-dimensional lattice model, we studied the effect
of ghost forces for dynamic problems. Ghost forces may arise for dynamic coupling
methods that were derived from energy approximation and Hamilton's principle \cite{AbBrBe98, Shenoy1999,Rodney03, XiBe04}. In our study, based on
an approximate model, we show that the error develops rather quickly. On the \(\mc{O}(1)\) time scale, the
error is observed in the entire domain, and at the interface, the gradient of the error is
 \(\mc{O}(1)\). Therefore, the influence of the error is more significant than that
 of the static case. It also suggests that the issue of the ghost force may even be more severe than the artificial reflections at the interface. 
 The analysis for the original dynamic QC model requires a different method, and it will be pursued in our future works.
\begin{appendix}
\section{Proof of Lemma~\ref{lema:altergen}}
The proof of this lemma is essentially the same with that of Lemma~\ref{lema:genest1}.

\noindent{\em Proof\;}
For $1\le\nu<\ga$, there exists $y_{\nu}\in(0,\pi/2)$ such that
$F_{\nu}(y_{\nu})=0$ with $\cos y_{\nu}=(\nu-\varrho)/\ga$. For any
$\eta\in \Lr{0,\min(y_\nu,\pi/2-y_\nu)}$
that will be chosen later, we write
\[
\int_0^{\pi/2}\vphi(y)e\Lr{G_{\nu}(y)}\dy=\Lr{\int_0^{y_{\nu}-\eta}+
\int_{y_{\nu}+\eta}^{\pi/2}+\int_{y_{\nu}-\eta}^{y_\nu+\eta}}
\vphi(y)e\Lr{G_{\nu}(y)}\dy.
\]
We deal with the three integrals separately. Using Lemma~\ref{lemma:firsttest}
with $\lam=\abs{F_{\nu}(y_{\nu}-\eta)}$,  we obtain
\[
\labs{\int_0^{y_{\nu}-\eta}\vphi(y)e(G_\nu(y))\dy}
\le\dfrac{1}{\pi\abs{F_{\nu}(y_{\nu}-\eta)}},
\]
and
\begin{align*}
\abs{F_{\nu}(y_{\nu}-\eta)}&=F_{\nu}(y_{\nu}-\eta)
=\dfrac{N\ga}{\pi}\dfrac{\cos(y_{\nu}-\eta)-\cos y_{\nu}}{\sin (y_{\nu}-\eta)}\\
&=\dfrac{2N\ga}{\pi}\dfrac{\sin(y_{\nu}-\eta/2)\sin(\eta/2)}{\sin (y_{\nu}-\eta)}\\
&\ge\dfrac{2N\ga}{\pi}\sin(\eta/2)\\
&\ge\dfrac{2N\ga\eta}{\pi^2},
\end{align*}
where we have used Jordan's inequality~\eqref{eq:jordan} in the last step.
This gives
\[
\labs{\int_0^{y_{\nu}-\eta}
\vphi(y)e\Lr{G_{\nu}(y)}\dy}\le\dfrac{\pi}{2N\ga\eta}.
\]

Using Lemma~\ref{lemma:firsttest} again
with $\lam=\abs{F_{\nu}(y_{\nu}+\eta)}$, we have
\[
\labs{\int_{y_{\nu}+\eta}^{\pi/2}\vphi(y)e(G_\nu(y))\dy}
\le\dfrac{1}{\pi\abs{F_{\nu}(y_{\nu}+\eta)}}.
\]
Furthermore,
\begin{align*}
\abs{F_{\nu}(y_{\nu}+\eta)}&=-F_{\nu}(y_{\nu}+\eta)
=\dfrac{N\ga}{\pi}\dfrac{\cos y_{\nu}-\cos(y_{\nu}+\eta)}{\sin(y_{\nu}+\eta)}\\
&=\dfrac{2N\ga}{\pi}\dfrac{\sin(y_{\nu}+\eta/2)\sin(\eta/2)}{\sin(y_{\nu}+\eta)}\\
&\ge\dfrac{2N\ga\eta}{\pi^2}\dfrac{\sin(y_{\nu}+\eta/2)}{\sin(y_{\nu}+\eta)},
\end{align*}
which yields
\[
\labs{\int_{y_{\nu}+\eta}^{\pi/2}
\vphi(y)e\Lr{G_{\nu}(y)}\dy}
\le\dfrac{\pi\sin(y_{\nu}+\eta)}{2N\ga\eta\sin(y_{\nu}+\eta/2)}.
\]
If \(y_\nu \in (0, \pi/4]\), we would require that \(\eta \in (0, y_\nu)\).
We will have, \(\sin (y_\nu + \eta)\le \sin 2y_\nu\le 2 \sin y_\nu\), and
\(\sin y_\nu<\sin(y_\nu+\eta/2)\) since \(y_\nu<y_\nu+\eta/2<2 y_\nu\le\pi/2\).
The bound
for the above integral is simplified to
\[
\labs{\int_{y_{\nu}+\eta}^{\pi/2}
\phi(y)e\Lr{G_{\nu}(y)}\dy}\le \dfrac{\pi}{N\ga\eta}.
\]
A trivial bound for the remaining integral yields
\[
\labs{\int_{y_{\nu}-\eta}^{y_\nu+\eta}
\vphi(y)e\Lr{G_{\nu}(y)}\dy}\le 2\eta.
\]
Summing up all the above estimates, we obtain
\begin{equation}\label{eq:basicsum}
\labs{\int_0^{\pi/2}
\vphi(y)e\Lr{G_{\nu}(y)}\dy}\le\dfrac{3\pi}{2N\ga\eta}+2\eta.
\end{equation}
Taking $\eta=(N\ga)^{-1/2}\sin y_{\nu}$, which is less than $y_\nu$, we obtain
\[
\labs{\int_0^{\pi/2}
\vphi(y)e\Lr{G_{\nu}(y)}\dy}\le \dfrac{3\pi}{2\sqrt{N \ga}\sin y_\nu}
+\dfrac{2\sin y_\nu}{\sqrt{N \ga} } \le\frac{3\pi}{\sqrt{N \ga}\sin y_\nu} .
\]

Now, if \(y_\nu \in (\pi/4,\pi/2]\), we would require that \(\eta\in (0, \pi/2-y_\nu)\).
We will have \(\sin(y_\nu+\eta)>\sin y_\nu\) since $y_\nu<y_\nu+\eta/2<y_\nu+\eta\le\pi/2$, and
\(\sin(y_\nu+\eta)\le 1\le\sqrt{2}\sin y_\nu\) since \(\sin y_\nu \ge 1/\sqrt2\). We bound the second integral as
\[
\labs{\int_{y_{\nu}+\eta}^{\pi/2}
\vphi(y)e\Lr{G_{\nu}(y)}\dy}\le \dfrac{\sqrt2\pi}{N\ga \eta},
\]
This yields
\[
\labs{\int_0^{\pi/2}
\vphi(y)e\Lr{G_{\nu}(y)}\dy}\le\dfrac{(1+\sqrt2)\pi}{2N\ga\eta}+2\eta.
\]
In this case, we might choose \(\eta=\dfrac12(N\ga)^{-1/2}\)  provided that
\begin{equation}\label{eq:cond}
\eta<\pi/2-y_\nu.
\end{equation}
With such choice of $\eta$, we have
\[
\labs{\int_0^{\pi/2}
\vphi(y)e\Lr{G_{\nu}(y)}\dy}\le\dfrac{3\pi}{\sqrt{N\ga}}.
\]
This choice of $\eta$ is feasible since
\[
\nu-\varrho\ge 1/2>(1/2)\sqrt{\ga/N}=\ga\eta,
\]
which yields \(\eta<\cos y_\nu\le\pi/2-y_\nu\), this  gives~\eqref{eq:cond} and completes the proof.
\section{Proof of Lemma~\ref{lemma:maxorder1}}
\begin{proof}
There exists $y_\nu\in(0,\pi/2)$ such that $F_\nu(y_\nu)=0$, and
\[
\cos y_\nu=\dfrac{\nu-\varrho}{\ga}=\dfrac{\ga-\fraction{\ga+\varrho}}{\ga}.
\]
Using the elementary inequality,
\begin{equation}\label{eq:cos}
1-\dfrac{2x}{\pi}\le\cos x\le 1-\dfrac{x^2}{\pi}\qquad x\in[0,\pi/2].
\end{equation}
we obtain
\begin{equation}\label{eq:zero}
\dfrac{\pi}2\dfrac{\fraction{\ga+\varrho}}{\ga}\le y_\nu\le\sqrt{\dfrac{\pi\fraction{\ga+\varrho}}\ga}.
\end{equation}
Proceeding along the same line that leads to~\eqref{eq:basicsum}, we get for any
$\eta\in(0,y_\nu)$,
\[
\labs{\int_0^{\pi/2}\varphi(y)e(G_\nu(y))\dy}\le\dfrac{3\pi}{N\ga\eta}+2\eta.
\]
We take $\eta=(N\ga)^{-1/2}$, which yields
\[
\labs{\int_0^{\pi/2}\varphi(y)e(G_\nu(y))\dy}\le
\dfrac{3\pi+2}{\sqrt{N\ga}}\le\dfrac{4\pi}{\sqrt{N\ga}}.
\]
If $\fraction{\ga+\varrho}$ satisfies
\[
\dfrac{\fraction{\ga+\varrho}}{\ga}>\dfrac{2}{\pi}(N\ga)^{-1/2},
\]
then using the left hand side of~\eqref {eq:zero}, we obtain $\eta\in(0,y_\nu)$.

On the other hand, if
\[
\dfrac{\fraction{\ga+\varrho}}{\ga}\le\dfrac2{\pi}(N\ga)^{-1/2},
\]
then letting $\del=2(N\ga)^{-1/4}$, we have
\begin{equation}\label{eq:cond1}
\dfrac{\fraction{\ga+\varrho}}{\ga}\le\dfrac{\del^2}{2\pi}\le\dfrac{1-\cos\del}2,
\end{equation}
where we have used the right hand side inequality
of~\eqref{eq:cos}. It follows from
the above inequality and the right hand side of~\eqref{eq:zero} that
\(
y_\nu\le\del/\sqrt2<\del.
\)
Using Lemma~\ref{lemma:firsttest} again
with $\lam=\abs{F_\nu(\del)}$, we get
\[
\labs{\int_{\del}^{\pi/2}\varphi(y)e(G_\nu(y))\dy}\le
\dfrac{1}{\pi\abs{F_\nu(\del)}}.
\]
We estimate the contribution of the complementary portion of the integral trivially:
\[
\labs{\int_0^{\del}\varphi(y)e(G_\nu(y))\dy}\le\int_0^\del y\dy=\dfrac{\del^2}2.
\]
On adding the above estimates we deduce that
\[
\labs{\int_0^{\pi/2}\varphi(y)e(G_\nu(y))\dy}\le
\dfrac{\del^2}2+\dfrac{1}{\pi\abs{F_\nu(\del)}}.
\]
A direct calculation gives
\[
\abs{F_\nu(\del)}=\dfrac{N\ga}\pi\dfrac{\cos y_\nu-\cos\del}{\sin\del}
=\dfrac{N\ga}\pi\dfrac{1-\cos\del-\fraction{\ga+\varrho}/\ga}{\sin\del}.
\]
Using~\eqref{eq:cond1}, we obtain
\[
\abs{F_\nu(\del)}\ge\dfrac{N\ga}{2\pi}\dfrac{1-\cos\del}{\sin\del}
=\dfrac{N\ga}{2\pi}\tan\dfrac{\del}2\ge\dfrac{N\ga\del}{4\pi}.
\]
This gives
\[
\labs{\int_0^{\pi/2}\varphi(y)e(G_{\nu}(y))\dy}
\le\dfrac{\del^2}2+\dfrac{4}{N\ga\del}= 2(N\ga)^{-1/2}+2(N\ga)^{-3/4}\le 4(N\ga)^{-1/2}.
\]
Finally, we have
\[
\labs{\int_0^{\pi/2}\varphi(y)e(G_{\nu}(y))\dy}
\le\max(4\pi(N\ga)^{-1/2},4(N\ga)^{-1/2})
=4\pi(N\ga)^{-1/2}.
\]
\end{proof}
\end{appendix}
\bibliographystyle{siam}
\bibliography{dyn}
\end{document}